\documentclass[11pt]{article}

\usepackage[english]{babel}
\usepackage{amsmath,amssymb,amsthm}
\usepackage{enumerate}
\usepackage[affil-it]{authblk}
\usepackage{xcolor}
\usepackage{mathrsfs}
\usepackage{verbatim}
\usepackage{url}
\usepackage[draft]{todonotes} % notes shown

\newtheorem{theorem}{Theorem}[section]
\newtheorem{proposition}[theorem]{Proposition}
\newtheorem{lemma}[theorem]{Lemma}
\newtheorem{corollary}[theorem]{Corollary}

\theoremstyle{definition}
\newtheorem{definition}[theorem]{Definition}

\newtheorem*{remark}{Remark}

\newcommand{\mb}[1]{\mathbf{#1}}
\newcommand{\mc}[1]{\mathcal{#1}}
\newcommand{\mcr}[1]{\mathscr{#1}}
\newcommand{\mP}{\mathcal{P}}

\newcommand{\R}{\mathbb{R}}
\newcommand{\Z}{\mathbb{Z}}
\renewcommand{\P}{\mathbb{P}}

\newcommand{\ip}[2]{\left< #1, #2 \right>}

\newcommand{\bd}{\partial}

\newcommand{\1}[1]{\mathbf{1}\left \{ #1 \right \}}
\newcommand{\Span}{\operatorname{Span}}
\newcommand{\supp}{\mathbf{\operatorname{supp}}}
\newcommand{\zero}{\mathbf{\operatorname{zero}}}
\newcommand{\Vol}{\operatorname{Vol}}

\title{On the Geometry of the Last Passage Percolation Problem}

\author{Tom Alberts \\
\texttt{alberts@math.utah.edu} \\
University of Utah
\and
Eric Cator \\
\texttt{e.cator@science.ru.nl} \\
Radboud University
}

\date{\today}

\begin{document}

\maketitle

\begin{abstract}
We analyze the geometrical structure of the passage times in the last passage percolation model. Viewing the passage time as a piecewise linear function of the weights we determine the domains of the various pieces, which are the subsets of the weight space that make a given path the longest one. We focus on the case when all weights are assumed to be positive, and as a result each domain is a pointed polyhedral cone. We determine the extreme rays, facets, and two-dimensional faces of each cone, and also review a well-known simplicial decomposition of the maximal cones via the so-called order cone. All geometric properties are derived using arguments phrased in terms of the last passage model itself. Our motivation is to understand path probabilities of the extremal corner paths on boxes in $\Z^2$, but all of our arguments apply to general, finite partially ordered sets.
\end{abstract}

\section{Introduction}

Last passage percolation is a well-studied model in probability theory that is simple to state but notoriously difficult to analyze. In recent years it has been shown to be related to many seemingly unrelated things: longest increasing subsequences in random permutations, eigenvalues of random matrices, long-time asymptotics of solutions to stochastic partial differential equations, and much more. All of these problems are of great interest due to the asymptotic behavior of various related statistics, neither of which are predicted by the classical strong law of large numbers or central limit theorem. The last passage model has been a particularly fertile ground for exploring this new frontier of probability theory due to its rich \textit{solvability} structure. For certain choices of the random inputs the last passage model can be analyzed exactly, through various connections with representation theory of the symmetric group and rings of symmetric polynomials. 

We briefly recall the setup of the last passage percolation model on $\Z^2$. Consider the box of integer points in $\Z^2$ with lower left corner at $(1,1)$ and upper right corner at $(m,n)$, where $m, n > 1$. At each of the $m \cdot n$ integer points $(i,j)$ we place a random variable $\omega(i,j)$ (a weight). The variables are typically assumed to be independent and identically distributed (iid) across points, and in this paper we will assume that they are always positive. We then consider the set $\mP(m,n)$ of \textit{up-right} paths from $(1,1)$ to $(m,n)$, an up-right path being one whose steps are always either $(1,0)$ or $(0,1)$. To each $\gamma \in \mP(m,n)$ we assign a random \textit{length} $\ell(\gamma)$ that is the sum of the $\omega$ along the path, i.e.
\[
\ell(\gamma) = \sum_{(i,j) \in \gamma} \omega(i,j).
\] 
In a lot of the literature on last passage percolation, either the start weight or the end weight of the path is left out in the length, so that concatenating paths is easier. It turns out that in our description, it is more convenient to consider all weights. Last passage percolation studies the maximal length over all paths, also known as the \textit{passage time}:
\begin{align*}
G(m,n) := \max_{\gamma \in \mP(m,n)} \ell(\gamma) = \max_{\gamma \in \mP(m,n)} \sum_{(i,j) \in \gamma} \omega(i,j).
\end{align*}
The passage time $G(m,n)$ is itself a random variable but its statistical distribution (or law) is very complicated. For any \textit{fixed} path $\gamma$, the law of the length $\ell(\gamma)$ is well understood by the Strong Law of Large Numbers and the Central Limit Theorem. The maximum length, however, is determined by the joint law which describes the statistics of the entire collection of random lengths, and the complicating feature is that there is a very strong \textit{correlation} between these different lengths. Whenever two paths share common vertices the random numbers at those vertices both contribute to their lengths, and so knowing the length of one path gives information about the length of the other. The more two paths intersect, the greater the correlation between their random lengths, and since there are $(m+n)!/(m!n!)$ paths but only $m \cdot n$ vertices the correlation effects are significant. 

Remarkably though, these correlation effects can be overcome when the choice of the input weights is assumed to be iid across vertices $(i,j)$ and coming from either the Bernoulli, geometric or exponential distribution. In these cases \textit{exact} formulas can be computed for the distribution function of $G(m,n)$. The formulas are somewhat complicated, however, and typically involve determinants of linear operators on the sequence space $\ell^2$, with the operator determined by certain families of orthogonal polynomials. Nonetheless, the formulas are somewhat explicit and tractable enough to perform asymptotic analysis as $m, n \to \infty$. Three very important and well known such asymptotic results are the following: 
\begin{itemize}
\item the \textit{limit shape} (the almost sure, non-random limit of $G(\lfloor nx \rfloor, \lfloor ny \rfloor)/n$ as $n \to \infty$, as a function of $x$ and $y$, whose existence follows from Kingman's subadditive ergodic theorem),
\item the magnitude and distribution of the fluctuations of the passage time $G(\lfloor n x \rfloor, \lfloor ny \rfloor)$ as $n \to \infty$ (the growth of the fluctuations being $n^{1/3}$ and the convergence of the centered and appropriately normalized passage time to the Tracy-Widom law),
\item and the magnitude of the transversal fluctuations of the maximal path away from the diagonal (the maximal path from $(1,1)$ to $(n,n)$ is thought to go distance $n^{2/3}$ away from the main diagonal, and is known to do so in certain solvable models).
\end{itemize}
The limit shape results are originally due to exact bijections between LPP and the TASEP process \cite{rost:particle_process, aldous_diaconis:particle_process, Timo:increasing_points}, or for stationary models that exist for certain special weight distributions \cite{OCY:burke, CG:Hammersley, BCS:cube_root, Timo:polymer}. More recent work \cite{GRS:cocycles_LPP, GRS:cocycles_CGM} provides variational formulas for the limit shape for very general weight distributions in terms of infinite dimensional objects called cocycles, although obtaining explicit results for these formulas is generally difficult. Exact Tracy-Widom limits for fluctuations are originally based on connections with generalized permutations and the Robinson-Schensted-Knuth algorithm \cite{johansson:generalized_perm}, often based on ideas from random matrix theory (see also \cite{PS:exact_solution_bernoulli, GeoOrt:bernoulli} in the Bernoulli case). In some cases these results have been re-understood through different means \cite{johansson:multi_dim_MC}, but in general all methods to date require a special choice of random input (see also \cite{Corwin:macdonald_review, Corwin:exactly_solving} for further references).

Nonetheless, it is widely believed that there is a certain \textit{universality} aspect to the last passage model. This specifically refers to the distribution of the fluctuations of the passage time $G(\lfloor nx \rfloor, \lfloor ny \rfloor)$ as $n \to \infty$, which is believed to be the same Tracy-Widom law for a wide class of random inputs, not just the special cases mentioned above. This is analogous to the Central Limit Theorem for sums of iid random variables, where the fluctuations of the sum follow the Gaussian distribution for a very broad class of input variables. While universality in the Central Limit Theorem is now understood via many different techniques and proofs, less progress has been made for universality results of the last passage model.

This paper explores a possible method for studying various aspects of the last passage model using tools from combinatorics and geometry. The main idea is to embed the model into a suitable high-dimensional space, determined by the random input weights, and in this space study the geometry of the last passage problem. The basic setup is relatively simple. For any fixed path $\gamma$ its length $\ell(\gamma)$ is clearly a linear function of the weights $\omega$, and therefore the passage time $G(m,n)$ is piecewise linear. The main purpose of this article is to determine the geometry of the domains of the pieces, each one of which corresponds to a different path. The main strength of this approach is that it is purely geometric, with no probabilistic input at all until a measure is put on the space of weights. This flexibility allows one to study many different types of random inputs with the same underlying geometric framework, and it is our hope that it will allow for a new conceptual framework for the last passage problem while at the same time shedding new light onto previously solvable models. Our interest in this approach was primarily driven by one simple question: among all paths in $\mP(n,n)$ (take $m = n$ for simplicity), which one is the most likely to be the maximizer? Even in the exactly solvable cases this does not seem to be an easy question to answer, as we explain later in Section \ref{sec:exp_weights}. While much attention has been paid to the paths with transversal fluctuations $n^{2/3}$, in particular the recent work \cite{DOV:directed_landscape} proves the existence of scaling limit for these objects in terms of the so-called Airy sheet \cite{CQR:KPZ_fixed_point}, less attention has been paid to the more extreme paths. We are quite confident that the most likely maximal path is the extreme one that goes straight up from $(1,1)$ to $(1,n)$ and then straight right from $(1,n)$ to $(n,n)$ (or its symmetric version that goes right and then up). We do not have a proof but the intuition is straightforward: the weights that are picked up by the extreme path are shared by a relatively small number of other paths, and therefore the extremal path should have a much larger portion of the environment space in which it is longest. For example, the extremal path gets the weight at $(1,n)$ entirely to itself. In contrast, the paths going through the interior share the weights they pick up with \textit{many} other paths, meaning each individual path has a hard time distinguishing itself as the longest. In fact, we expect that the probability of the extremal path being longest is substantially larger (in $n$) than the probability of the middle path being longest (the middle path being the one that alternates between up and right steps). This heuristic fits with the expectation that the transversal fluctuations are larger than the $n^{1/2}$ magnitude obtained by the uniform measure on paths. In fact some sort of behavior of this type seems necessary to obtain superdiffusive fluctuations, although on its own it does not explain why the magnitude of the fluctuations should be precisely $n^{2/3}$. We expect that the $n^{2/3}$ corresponds to the region where the low probabilities for the ``middle paths'' balances out the fact that the bulk of the paths are in the middle. In other words, even though we believe that the extremal path (which has transversal fluctuation of order $n$) is the mode of the path distribution, a typical sample from the path distribution has transversal fluctuations of much smaller magnitude because there are so many more paths there.

In the last section of the paper we discuss some other results that we believe follow from this intuition, such as a negative correlation result between the location of the maximizer and the path length itself. The present paper comes from a desire to solidify our intuition by understanding more about the structure of the parts of environment space that makes a given path the longest. We also hope it will help to separate out how much of the expected universal behavior is due to the geometry of the last passage time function and how much is due to the particular probability distribution on the weights. One advantage of our framework is that it extends beyond the traditional study of LPP on $\Z^2$. In fact all that is required is a notion of directedness, which allows us to carry out the analysis on general finite posets.  

\subsection*{General Setup and Main Results}

Although we are largely motivated by the last passage problem on $\Z^2$, our approach assumes nothing other than the paths being \textit{directed}. On $\Z^2$ this is forced by the assumption that paths are up-right (and hence not allowed to go backwards), but in fullest generality we can force a direction by studying the problem on an arbitrary partially ordered set (poset). This has the advantage of allowing for different correlation structures among path lengths, which is determined by the structure of the underlying poset as follows.

Let $(P, \leq)$ be a finite poset. The general last passage problem will be considered on $P$, which we often think of in terms of its Hasse diagram, and so we will commonly refer to the elements of $P$ as vertices. We will assume throughout that $P$ is connected, meaning that its Hasse diagram is connected as a graph, since otherwise we may consider the problem individually on the different connected components. The cover relations of the poset will be denoted by $\lessdot$, where we recall that for $\mb{v}, \mb{w} \in P$, $\mb{v} \lessdot \mb{w}$ means that $\mb{v} < \mb{w}$ and there is no $\mb{u} \in P$ such that $\mb{v} < \mb{u} < \mb{w}$. So there is an edge in the Hasse diagram of $P$ from $\mb{v}$ to $\mb{u}$ iff $\mb{v}\lessdot \mb{u}$. On a general poset the paths of the last passage problem are the \textit{maximal chains} of $P$, the set of which we denote by $\Pi_P$. Recall that a maximal chain is an ordered subset $\{\mb{v}_1, \ldots, \mb{v}_n \}$ of $P$ such that $\mb{v}_1 \lessdot \mb{v}_2 \lessdot \ldots \lessdot \mb{v}_n$ and there are no elements $\mb{u}$ or $\mb{w}$ such that $\mb{u} \lessdot \mb{v}_1$ or $\mb{v}_n \lessdot \mb{w}$. Intuitively we see that this corresponds to all nearest-neighbor paths in the Hasse diagram of $P$ that are as ``long'' as possible.

For the (positive weight) last passage problem on $P$ we place a weight $\omega(\mb{v})\in\R_+ = [0, \infty)$ on each element $\mb{v} \in P$. The vector $\omega \in \R_+^P$ is collectively referred to as the weight, and the length of each element of $\pi \in \Pi_P$ is the the sum of the weights along the path:
\[
\ell(\pi) := \sum_{\mb{v} \in \pi} \omega(\mb{v}).
\]
Note that we can naturally associate each path $\pi \in \Pi_P$ to a vector in $\R_+^P$ (which we also call $\pi$) via $\pi(\mb{v}) = \1{\mb{v} \in \pi}$, where $\mb{1}$ is the indicator function. Via this association we have that the length is simply the standard inner product between the path and the weight vector, i.e.
\[
\ell(\pi) = \langle \omega, \pi \rangle.
\]
The \textit{passage time} of the poset $P$, under the weight vector $\omega$, is the largest length of all possible paths, i.e.
\[
G_P = G_P(\omega) := \max_{\pi \in \Pi_P} \langle \omega, \pi \rangle.
\]
More generally we may consider the vector of passage times determined by the weight vector $\omega$, which encodes the length of the longest path up to each given vertex and is defined as
\[
G_P(\mb{v}) = G_P(\mb{v}; \omega) := \max_{\pi \in \Pi_P(\mb{v})} \langle \omega, \pi \rangle,
\]
where $\Pi_P(\mb{v})$ is the set of all maximal chains in the subposet of elements below $\mb{v}$ (the so-called lower set of $\mb{v}$, see below for a definition), extended to $\R^P$ by adding zeros. Then clearly
\[
G_P = \max_{\mb{v} \in P} G_P(\mb{v}).
\]
The collection of passage times $G_P(\mb{v})$ can also be built up from the weight vector $\omega$ via the recursion
\begin{align}\label{eqn:recursion}
 G_P(\mb{v}) = \omega(\mb{v}) + \max_{\mb{u} : \mb{u} \lessdot \mb{v}} G_P(\mb{u})
\end{align}
with the ``initial condition'' $G_P(\mb{v}) = \omega(\mb{v})$ if $\mb{v}$ is a minimal element of $P$. Conversely, given the vector of passage times $G_P(\mb{v})$ this recursion can be inverted to solve for the corresponding weight vector $\omega$ via
\[
 \omega(\mb{v}) = G_P(\mb{v}) - \max_{\mb{u} : \mb{u} \lessdot \mb{v}} G_P(\mb{u}),
\]
again with $G_P(\mb{v}) = \omega(\mb{v})$ for $\mb{v}$ minimal. Regardless of how $G_P$ is constructed, for $P$ fixed and $\omega$ allowed to vary, this definition implies that $G_P$ is a piecewise linear function of $\omega$, and the main purpose of this article is to determine the regions on which the function is equal to each of the various linear maps that define it. Since in this case the maps are defined by the paths $\pi$ there is a natural region in $\R_+^P$ associated to each path: the set of weight vectors $\omega$ that give path $\pi$ the longest weight. More precisely, this is the set
\begin{align}\label{defn:max_path_set}
\mc{C}(\pi) := \left \{ \omega \in \R_+^P : G_P(\omega) = \langle \omega, \pi \rangle \right \}
		  = \left \{ \omega \in \R_+^P : \langle \omega, \pi \rangle \geq \langle \omega, \pi' \rangle \textrm{ for all } \pi' \in \Pi_P  \right \}.
\end{align}
From this definition and especially the second equality we immediately see that each set $\mc{C}(\pi)$ is a \textit{polyhedral cone}, namely a finite intersection of half-spaces of $\R^P$. The inequalities defining the half-spaces are those of the form $\langle \omega, \pi - \pi' \rangle \geq 0$, with $\pi$ and $\pi'$ regarded as vectors in $\R_+^P$, but also those inequalities implicitly given by the condition that the cone is a subset of $\R_+^P$. The latter is equivalent to saying that $\omega(\mb{v}) \geq 0$ for all $\mb{v} \in P$, which simply increases the number of half-spaces that define the cone.

As with all polyhedral cones the sets $\mc{C}(\pi)$ are both convex and invariant under positive scaling, since the half-spaces that define them are also and these properties are preserved under intersection. Both properties also follow from their interpretation via the last passage model, since if two weight vectors make the same path maximal then clearly so does their sum and any positive scalar multiple.

Beyond the fact that the maximal sets are polyhedral cones, a more detailed description of the structure of the sets is required to perform any meaningful analysis. There are two common descriptions of a polyhedral cone: via the set of half-spaces that bound it (the \textit{H-decomposition}), or via the extreme rays that span it (the \textit{V-decomposition}). For a polyhedral cone there are at most finitely many half-spaces and extreme rays that define it, and in this article we will determine both for each given path $\pi$. It turns out that both descriptions have a very beautiful structure, and moreover can be determined solely by working with their description in terms of the last passage model. The V-decomposition is already known in \cite{stanley:two_poset_polytopes} but our argument is different in that it is phrased in terms of the last passage model. To the best of our knowledge our determination of the H-decomposition is new, and we regard it as the most significant of our results. In both cases the idea is that the geometry of each maximal set $\mc{C}(\pi)$ is naturally encoded in the poset $P$, and our arguments are based on a comparison of $\pi$ to the other paths in $\Pi_P$. In an intuitive sense we are analyzing the ability of the other paths to compete with $\pi$ to be the maximal one. We will prove the following two main theorems on the $H$ and $V$ decompositions of the sets $\mc{C}(\pi)$:

\begin{theorem}[$H$-decomposition of maximal sets]\label{thm:faces}
For each path $\pi \in \Pi_P$ the \textbf{minimal} set of inequalities that define the cone $\mc{C}(\pi)$ are those of the form:
\begin{enumerate}[i)]
\item $\omega(\mb{v}) \geq 0$ for $\mb{v} \in P \backslash \pi$,
\item $\omega(\mb{v}) \geq 0$ for $\mb{v} \in \pi$ but not a corner of $\pi$,
\item $\langle \omega, \pi-\pi' \rangle \geq 0$ for paths $\pi' \in \Pi_P \backslash \{ \pi \}$, whose disorder graph with $\pi$ is connected.
\end{enumerate}
\end{theorem}

The second and third conditions require further definition, which we give next, but are motivated by their meaning on $\Z^2$, which for the second condition is geometrically intuitive and in the third condition means that $\pi$ and $\pi'$ form at most a single loop. See the remark below for more. On general posets they mean the following:

\begin{definition}
A vertex $\mb{v} \in P$ is a corner of a path $\pi \in \Pi_P$ if $\mb{v} \in \pi$ and there exists another path $\pi' \in \Pi_P$ such that $\pi\setminus{\mb{v}}\subset \pi'$ and $\mb{v}\not\in \pi'$.
\end{definition}

The notion of a corner of a path can be easily visualized in the Hasse diagram of the poset. 

\begin{definition}\label{defn:disorder_graph}
Fix $\pi, \pi' \in \Pi_P$. The \emph{disorder graph} of $\pi$ and $\pi'$, denoted by $\Delta(\pi,\pi')$, has as its vertex set the symmetric difference $\pi\triangle\pi'$ between $\pi$ and $\pi'$ seen as subsets of the poset (so $(\pi \backslash \pi') \cup (\pi' \backslash \pi$)). There is an edge between $\mb{u}$ and $\mb{v}$ precisely when $\mb{u}$ and $\mb{v}$ are out of order.
\end{definition}

An important property of $\Delta(\pi,\pi')$ is that it constitutes a bipartite graph, with one part consisting of the vertices belonging to $\pi$ and the other part consisting of the vertices belonging to $\pi'$. The bipartiteness follows because two vertices belonging to the same path are always in order.

Note that the statement of Theorem \ref{thm:faces} is that this is the minimal set of inequalities needed to define the cone, so that removing any one of them would lead to a larger set than $\mc{C}(\pi)$. These inequalities define the \textit{facets} of the cone, the co-dimension one boundary sets of $\mc{C}(\pi)$. Note that \eqref{defn:max_path_set} already defines $\mc{C}(\pi)$ via these various inequalities, but what Theorem \ref{thm:faces} amounts to showing is that many of these facets are \textit{redundant}. Reducing the inequalities to only the irredundant ones allows for a fuller analysis of the cone, and is usually required for computational algorithms.

\begin{remark}
For posets of the form $[1,m] \times [1,n]$ in $\Z^2$ the condition that the disorder graph $\Delta(\pi, \pi')$ is connected is equivalent to saying that $\pi$ and $\pi'$ form a single loop. That is, $\pi$ and $\pi'$ may start out the same, diverge for a while, and then recombine with each other, but after recombining cannot diverge again. Diverging more than once would mean that there are multiple loops between $\pi$ and $\pi'$, which is equivalent to saying that $\pi - \pi'$ can be written as the sum of the individual loops. Since each individual loop is already a face of $\mc{C}(\pi)$ the sum is redundant. It is also easy to see that $\pi$ and $\pi'$ forming multiple loops is equivalent to the disorder graph being disconnected, since between any two consecutive loops there is a subpath in $\pi \cap \pi'$ that connects the loops together. This subpath prevents the two loops from being connected in $\Delta(\pi, \pi')$.
\end{remark}

We prove Theorem \ref{thm:faces} for the H-representation in Section \ref{sec:faces}. In Section \ref{sec:two_dim_faces} we also describe a related object called the \textit{order graph}, which forms connections based on ordering relations between the supported vertices of two extreme rays of $\mc{C}(\pi)$. We use the order graph to determine when two extreme rays of $\mc{C}(\pi)$ form a two-dimensional boundary face of the cone, see Theorem \ref{thm:edges}.

To describe the extreme rays of the cones requires the notion of an \textit{antichain} of the poset $P$ and a particular geometric embedding of it, which we define next.

\begin{definition}
An \textit{antichain} of the poset $P$ is a subset of $P$ such that no two elements are in order. We will naturally embed an antichain $A \subset P$ into an element $a \in \R_+^P$ via
$a(\mb{v}) = \1{v \in A}$.
\end{definition}

In particular, an antichain can contain at most one vertex from a given path in $\Pi_P$, since by definition the elements along a path are in complete order with each other. This leads to the following theorem.

\begin{theorem}[$V$-decomposition of maximal sets]\label{thm:extreme_rays}
For each path $\pi \in \Pi_P$ the extreme rays of the cone $\mc{C}(\pi)$ are precisely the geometric embeddings of the antichains which intersect $\pi$ exactly once.
\end{theorem}

In other words, a vector in $\R_+^P$ is an extreme ray of $\mc{C}(\pi)$ if and only if there is one non-zero entry along the path $\pi$, and all other non-zero entries have the same value and are out of order with each other in the poset. We will canonically take the non-zero entry to be $1$, although by scaling invariance it could clearly be any positive value. Commonly we will use the notation:

\begin{definition}
For each fixed $\pi \in \Pi_P$ we let $\mcr{ER}(\pi)$ denote the set of extreme rays of the polyhedral cone $\mc{C}(\pi)$.
\end{definition}

The structure of extreme rays is essentially already stated by Stanley \cite{stanley:two_poset_polytopes} through what he calls the \textit{chain polytope}. See also the earlier works referenced within \cite{stanley:two_poset_polytopes}. The chain polytope can be formed by intersecting each maximal cone $\mc{C}(\pi)$ with the unit cube $[0,1]^P$ and then taking the union of what remains over all paths $\pi$. Our description of the extreme rays for each individual cone $\mc{C}(\pi)$ is not a very extensive refinement of Stanley's result, but our proof is different in that it is framed entirely in terms of the last passage model. The antichains turn out to be precisely the directions in which one can perturb the path lengths while keeping the longest path the longest, and this turns out to be the key argument in our proof. This is laid out in Section \ref{sec:extreme_rays}.

The explicit structure of $\mcr{ER}(\pi)$ also allows us to determine its size for certain types of posets, in particular for $[1,m] \times [1,n] \subset \Z^2$. See Theorem \ref{thm:extreme_rays_count}. For all but pathological posets and paths the number of extreme rays in $\mcr{ER}(\pi)$ is much greater than the dimension $|P|$ of the ambient space $\R^P$, meaning that the maximal cones $\mc{C}(\pi)$ are far from simplicial. Nonetheless it is possible to use the extreme rays in $\mcr{ER}(\pi)$ to give an explicit simplicial decomposition of each maximal cone $\mc{C}(\pi)$, without the need to introduce additional rays. 

\begin{theorem}\label{thm:simplicial_decomposition}
For each $\pi \in \Pi_P$ there is a decomposition of $\mc{C}(\pi)$ into disjoint simplicial cones (disjoint up to boundary intersections) such that the extreme rays of each simplicical cone only use elements from $\mcr{ER}(\pi)$.
\end{theorem}

This theorem can be found in Stanley \cite{stanley:two_poset_polytopes, stanley:enum_combin_ii} so we only explain it briefly in Section \ref{sec:simplicial_decomp}. On Young diagrams (which we regard as subposets of $\Z^2$) it is equivalent to using Young tableaux to partition the space into simplices. From this partitioning we obtain the following result:

\begin{corollary}
There exists functions $\Lambda_1, \ldots, \Lambda_{|P|} : \R_+^P \to \R$ such that
\[
G_P(\omega) = \sum_{i=1}^{|P|} \Lambda_i(\omega).
\]
\end{corollary}

The main purpose of this corollary is that it converts a complicated maximum of random variables into a sum of the same number of random variables. While sums are usually easier to handle, the mapping from $\omega$ to $\Lambda$ is piecewise linear and induces a complicated correlation structure on the $\Lambda_i$ random variables, even when the underlying $\omega$ distribution is nice. This representation of the passage time as a sum of random variables is equivalent to the corner growth representation of the last passage model \cite{timo:CGM_notes, Romik:book}, in which the elements of the poset are ``filled in'' at random times that obey the ordering of the poset. 
We briefly explain this connection towards the end of Section \ref{sec:simplicial_decomp}. In Section \ref{sec:exp_weights} we give some explanation of how the iid exponential distribution for the weight variables interacts nicely with the geometry of the last passage function; this gives some additional intuition into why the exponential distribution tends to produce the most precise results. In Section \ref{sec:uniform_weights} we describe how the geometrical description of the maximal cones can be used to give an alternative description of the passage time for iid Uniform$(0,1)$ weights, in terms of Stanley's order cone \cite{stanley:two_poset_polytopes}. Finally, in Section \ref{sec:open} we list a series of open problems that this work has led us to.
\newline 

\noindent
\textbf{Acknowledgments:} Tom Alberts gratefully acknowledges the contributions of Bryant Lin, who performed computer work that laid the groundwork for this project as part of a Summer Undergraduate Research Fellowships (SURF) program at Caltech . Alberts and Lin thank the SageMath project (\texttt{www.sagemath.org}) for their excellent software which enabled us to explore and verify properties of the maximal cones. Alberts and Cator thank Leonid Petrov for helpful discussions. Alberts was supported by Simons Collaboration Grant 351687 and National Science Foundation grants DMS-1811087 and DMS-1715680.  

\section{Extreme Rays \label{sec:extreme_rays}}

In this section we concentrate on proving Theorem \ref{thm:extreme_rays}. We recall that for a vector $\omega$ to be an extreme ray of a polyhedral cone means that it can only be written as a (positive) weighted sum of (positive) multiples of itself, i.e. if $\omega, \omega_1, \omega_2$ are all in the same polyhedral cone then
\[
\omega = \alpha \omega_1 + \beta \omega_2 \textrm{ with } \alpha, \beta > 0 \implies \omega_1, \omega_2 \in \Span_+ \{ \omega \}.
\]
Another way of saying this is that the only linear subspace of directions in which one can move infinitesimally away from $\omega$ and still remain in the cone is $\Span \{ \omega \}$. We will use this type of argument throughout our analysis, which leads to the following definition:

\begin{definition}
For a path $\pi \in \Pi_P$ and a vector $\omega \in \R_+^P$, we define the \textit{perturbation space} of $\omega$ in the cone $\mc{C}(\pi)$ by
\[
\mc{D}_\pi(\omega) = \left \{ \sigma \in \R^P : \exists \, \epsilon > 0 \textrm{ such that } \omega \pm \epsilon \sigma \in \mc{C}(\pi) \right \},
\]
so long as $\omega \in \mc{C}(\pi)$. If $\omega \not \in \mc{C}(\pi)$ we set $\mc{D}_\pi(\omega) = \emptyset$.
\end{definition}

Note that as long as $\omega \in \mc{C}(\pi)$ then scale invariance of the cone implies that $\Span \{ \omega \} \subset \mc{D}_{\pi}(\omega)$. Furthermore, from this definition it is straightforward to verify:

\begin{lemma}
The perturbation space $\mc{D}_\pi(\omega)$ is a linear subspace of $\R^P$. Moreover, if $\omega \in \mc{C}(\pi)$ then it is an extreme ray of $\mc{C}(\pi)$ if and only if $\mc{D}_\pi(\omega) = \operatorname{Span}\{\omega\}$.
\end{lemma}

The previous lemma will be our key tool for proving Theorem \ref{thm:extreme_rays}. First we will show that all antichains of $P$ that intersect $\pi$ have only their span in their perturbation space, and then conversely that all weight vectors that make a given path maximal and have only their span in their pertubation space must be maximal. To this end we first note the following simplification.

\begin{remark}
Fix a path $\pi$. Then to determine the perturbation space of a vector $\omega \in \mathcal{C}(\pi)$ it is enough to consider only the non-zero entries of the vector that can be perturbed. Indeed, the zero entries can never be perturbed since necessarily the perturbation in either the positive or negative direction will take them out of $\R_+^P$, which violates that $\mathcal{C}(\pi)$ is a subset of $\R_+^P$. We will use this simple fact repeatedly so we define:
\end{remark}

\begin{definition}
For a vector $\omega \in \R^P$ we define the support of $\omega$ to be the subset of vertices on which $\omega$ is non-zero, i.e. $\supp(\omega) := \{ \mb{v} \in P : \omega(\mb{v}) \neq 0\}$. We let $\zero(\omega) = P \backslash \supp(\omega) = \{ \mb{v} \in P : \omega(\mb{v}) = 0\}$.
\end{definition}

\begin{proof}[Proof of Theorem \ref{thm:extreme_rays} -- Antichains are extreme rays]
Fix $\pi \in \Pi_P$ and suppose that $a$ is an antichain of $P$. The first observation is that all paths have length either $0$ or $1$ under $a$ since the vertices along a path are in order and the elements of $a$ are completely out of order, hence a given path can intersect $a$ at most once. If $a$ is non-zero at some vertex of $\pi$ then clearly $\pi$ has length $1$ under $a$ and therefore is maximal under $a$, i.e. $a \in \mc{C}(\pi)$.

Now let $\sigma \in \mc{D}_\pi(a)$. By the last remark we can assume that $\sigma$ has zero entries at all vertices where $a$ has zero entries, i.e. $\supp(\sigma) \subset \supp(a)$. Now suppose that $\sigma$ is not constant on $\supp(a)$. Choose $\epsilon>0$ such that $a-\epsilon\sigma\geq 0$ (this is possible since $a$ is strictly positive on $\supp(a)$). Then under the weight vector $a + \epsilon \sigma$ the maximal paths are those which were maximal under $a$ \textit{and} pass through vertices in $\supp(a)$ at which $\sigma$ achieves its maximal value, which by the non-constancy assumption is not all of $\supp(a)$. If $\pi$ is not one of these paths then it is no longer one of the longest, so by definition $a + \epsilon \sigma \not \in \mc{C}(\pi)$, and this holds for all $\epsilon > 0$. If $\pi$ is one of these paths then the non-constancy assumption means it cannot be longest under $a - \epsilon \sigma$, again for all $\epsilon$ small enough. Thus if $\sigma$ is not constant on $\supp(a)$, it cannot be in $\mc{D}_\pi(a)$, which completes the proof.
\end{proof}

To prove the opposite direction is relatively simple but slightly lengthier, so we break the proof into several smaller supporting results. First recall the following terminology:

\begin{definition}
For $A \subset P$ the lower set $L(A)$ of $A$ is the set of elements below $A$ in $P$, i.e. $L(A) = \{ \mb{u} \in P : \mb{u} \leq \mb{v} \textrm{ for some } \mb{v} \in A \}$. Similarly the upper set is $U(A) = \{ \mb{u} \in P : \mb{u} \geq \mb{v} \textrm{ for some } \mb{v} \in A \}$. We also define their boundaries $\bd L(A)$ and $\bd U(A)$ as the maximal and minimal elements of $L(A)$ and $U(A)$, respectively.
\end{definition}

Note that both $\bd L(A)$ and $\bd U(A)$ are antichains of $P$, and from this one immediately has that:

\begin{lemma}
A subset $A \subset P$ is an antichain of $P$ iff $\bd L(A) = A = \bd U(A)$.
\end{lemma}

We will use this lemma for the subset $\supp(\omega)$ determined by a weight vector $\omega \in \R^P$. In particular we use it to show that:

\begin{lemma}
If $\omega$ is an extreme ray of $\mc{C}(\pi)$ then necessarily $\supp(\omega)$ is an antichain of $P$.
\end{lemma}

\begin{proof}
For shorthand write $\bd L = \bd L(\supp(\omega))$ and $\bd U = \bd U(\supp(\omega))$. We will show that $\bd L = \bd U$. First observe that any path $\pi'$ can pass through at most one element from each of $\bd L$ and $\bd U$, since both are unordered antichains and the path is ordered. Suppose $\pi'$ is one of the paths which is longest under $\omega$. Define $S=\pi'\cap \supp(\omega)$. Since $S\subset \pi'$, it has a unique minimal element $\mb{v}$. If $\mb{v}\not\in\bd U$, there exists $\mb{w}<\mb{v}$ with $\omega(\mb{w})>0$ (since $\mb{v}\in U(\supp(\omega))$). Then we can construct a path $\pi''$ such that $\mb{w}\cup S\subset \pi''$, which would mean that $\pi''$ is longer than $\pi'$ under $\omega$; contradiction. This shows that $\pi'\cap \bd U\neq \emptyset$, and in a similar manner we can show that $\pi'\cap \bd L\neq \emptyset$. Hence for all such $\pi'$ we have that $\langle \pi', 1_{\bd L} \rangle = 1 = \langle \pi', 1_{\bd U} \rangle$, where $1_{\bd L}$ and $1_{\bd U}$ are the indicator functions of $\bd L$ and $\bd U$, respectively.

Now let $\epsilon_1 = \min \{ \langle \pi',\omega\rangle - \langle \pi'',\omega\rangle : \langle \pi',\omega\rangle > \langle \pi'',\omega\rangle \}$, $\epsilon_2 = \min \{ \omega(\mb{v}) : \mb{v}\in\supp(\omega)\})$ and $\epsilon=\min(\epsilon_1,\epsilon_2)$. Note that $\epsilon>0$. Define a vector $\sigma \in \R^P$ by
\[
\sigma = \frac{\epsilon}{2} (1_{\bd L} - 1_{\bd U}).
\]
Because $\epsilon\leq \epsilon_2$, we have that $\omega-\sigma\geq 0$ (because $\supp(\sigma) \subset \bd L \cup \bd U \subset \supp(\omega)$, so that $\sigma$ is zero at all vertices where $\omega$ is). We proceed by contradiction. Using from above that $\langle \pi', 1_{\bd L} \rangle = 1 = \langle \pi', 1_{\bd U} \rangle$ for \textit{any} $\pi'$ which is longest under $\omega$ (including $\pi$ itself), we have
\[
\langle \pi', \omega \pm \sigma \rangle = \langle \pi', \omega \rangle \pm \frac{\epsilon}{2}(\langle \pi', 1_{\bd L} \rangle) - \langle \pi', 1_{\bd U} \rangle) = \langle \pi', \omega \rangle,
\]
so that all paths which were longest under $\omega$ are still longest under $\omega \pm \sigma$ (here we use that $\epsilon\leq \epsilon_1$, so that second longest paths cannot overtake any of the longest paths). In particular $\pi$ itself is still a longest path, so $\omega \pm \sigma \in \mc{C}(\pi)$, which implies that $\sigma \in \mc{D}_{\pi}(\omega)$. But the assumption $\bd L \neq \bd U$ also gives that $\sigma \not \in \Span \{ \omega \}$, which is a contradiction to $\omega$ being an extreme ray.
\end{proof}

To complete the proof of Theorem \ref{thm:extreme_rays} it only remains to be shown that each extreme ray must be constant on its support.

\begin{proof}[Proof of Theorem \ref{thm:extreme_rays} -- Extreme rays must be constant on antichains] Let $\omega$ be an extreme ray of $\mc{C}(\pi)$. Then by the last lemma its support is an antichain of $P$. Thus $\pi$ can pass through at most one element of $\supp(\omega)$, but it must path through at least one since otherwise its length would be zero and it could not be maximal.

Let $\mb{v}$ be the element of $\supp(\omega)$ that $\pi$ passes through. If $\omega(\mb{v}) < \omega(\mb{w})$ for some $\mb{w} \in \supp(\omega)$ then $\pi$ could not have been longest under $\omega$ since any path that goes through $\mb{w}$ would be longer. If, on the other hand, $\omega$ achieves its maximal value at $\mb{v}$ then let $A = \{ \mb{w} \in \supp(\omega) : \omega(\mb{v}) > \omega(\mb{w}) \}$. Then under $\omega$ the maximal paths are those which pass through $\supp(\omega) \backslash A$. Let $\epsilon = \min \{ \omega(\mb{v}) - \omega(\mb{w}) : \mb{w} \in A \}$, which we note is strictly positive, and then define a vector $\sigma \in \R^P$ by
\[
\sigma = \frac{\epsilon}{3}(1_{\supp(\omega) \backslash A} - 1_A).
\]
Then any path $\pi'$ which was longest under $\omega$ is still longest under $\omega \pm \sigma$, and hence $\sigma \in \mc{D}_{\pi}(\omega)$. But if $A \neq \emptyset$ then $\sigma \not \in \Span \{ \omega \}$, and  this says $\mc{D}_{\pi}(\omega)$ is strictly larger than $\Span \{ \omega \}$. This contradicts that $\omega$ is an extreme ray of $\mc{C}(\pi)$.
\end{proof}

Finally, we end this section by proving a formula for the number of extreme rays of a maximal cone $\mc{C}(\pi)$ on the subposet $[1,m] \times [1,n]$ of $\Z^2$. 

\begin{theorem}\label{thm:extreme_rays_count}
Let $m, n > 1$ and $P = [1,m] \times [1,n]$ as a subposet of $\Z^2$ with the componentwise ordering. Write a path $\pi \in \Pi_P$ as the ordered collection of vertices $(u_i, v_i)$ with $(u_1,v_1) = (m,n)$, $(u_{m+n-1}, v_{m+n-1}) = (m,n)$, 
and $(u_{i+1}, v_{i+1}) - (u_i, v_i) \in \{ (0,1), (1,0) \}$. Then the number of extreme rays in $\mcr{ER}(\pi)$ is
\begin{align}\label{eqn:num_extreme_rays_per_path}
\sum_{i=1}^{m+n-1} \dbinom{n+u_i-v_i-1}{u_i-1} \dbinom{m - u_i + v_i - 1}{v_i-1}.
\end{align}
\end{theorem}

\begin{proof}
Any extreme ray of $[1,a] \times [1,b]$ can have at most one non-zero entry in each row and column, thus at most $a \wedge b$ non-zero entries overall. To construct extreme rays with exactly $k \geq 1$ non-zero entries do the following: independently choose subsets $A \subset \{ 1, \ldots, a \}$ and $B \subset \{1, \ldots, b \}$ with $|A| = |B| = k$, and from them form $k$ vertices by pairing the elements of $A$, sorted in increasing order, with the elements of $B$, sorted in decreasing order. By construction these $k$ vertices are all out of order and hence form an antichain. Conversely, given any antichain of with exactly $k \geq 1$ vertices the corresponding subsets $A$ and $B$ are determined uniquely. Therefore the block $[1,a] \times [1,b]$ has exactly
\[
J(a,b) = \sum_{k=1}^{a \wedge b} \dbinom{a}{k} \dbinom{b}{k}
\]
antichains. By Vandermonde's identity
\[
1 + \sum_{k=1}^{a \wedge b} \dbinom{a}{k} \dbinom{b}{k} = \sum_{k=0}^{a \wedge b} \dbinom{a}{k} \dbinom{b}{k} = \dbinom{a+b}{a}
\]
Now for the formula for $|\mcr{ER}(\pi)|$, recall that by Theorem \ref{thm:extreme_rays} every antichain must contain exactly one vertex along the path. The sum in \eqref{eqn:num_extreme_rays_per_path} partitions the elements of $\mcr{ER}(\pi)$ according to which vertex is included. Each such vertex $(u_i, v_i)$ naturally breaks the poset $[1,a] \times [1,b]$ into four quadrants, and any extreme ray containing $(u_i, v_i)$ in its support must have the rest of its non-zero entries in the northwest and southeast quadrants. More precisely, the support must be in the complementary set of $L((u_i,v_i)) \cup U((u_i,v_i))$. The northwest quadrant is precisely $[1, u_i-1] \times [v_i+1, n]$ and the southwest one is $[u_i+1,m] \times [1,v_i-1]$. The total number of extreme rays containing $(u_i,v_i)$ can then be broken into four distinct types: those with non-zero entries in both quadrants, those with non-zero entries in only one of the quadrants, and the single extreme ray supported only at $(u_i,v_i)$. Therefore the total number of extreme rays of $[1,m] \times [1,n]$ that have $(u_i, v_i)$ in their support is
\[
(1 + J(u_i-1,n-v_i))(1 + J(m-u_i, v_i-1))
\]
Combining this with Vandermonde's identity completes the proof.
\end{proof}

\section{Facets \label{sec:faces}}

In this section we prove Theorem \ref{thm:faces} on the facets of the maximal cones $\mc{C}(\pi)$, again using reasoning that is purely in terms of the last passage model. For each path $\pi \in \Pi_P$ we start with the definition \eqref{defn:max_path_set} of $\mc{C}(\pi)$ and determine which inequalities that define it are redundant and which are necessary. The necessary ones are precisely the facets of the cone.

To accomplish this we let $N_{\pi}$ be the set of normal vectors which describe the half-spaces defining $\mc{C}(\pi)$, i.e.
\[
N_{\pi} = P \cup \left \{ \pi - \pi' : \pi' \in \Pi_P \backslash \{ \pi \} \right \}.
\]
Note that we are considering the elements of $P$ as the basis vectors $\delta_{\mb{v}}$, $\mb{v} \in P$, in this case. Then for each $\mb{\eta} \in N_{\pi}$ we define $\mc{C}(\pi; \mb{\eta})$ to be the same polyhedral cone as $\mc{C}(\pi)$ but after removing the bounding hyperplane with normal $\mb{\eta}$, i.e.
\[
\mc{C}(\pi; \mb{\eta}) = \{ \omega \in \R^P : \langle \omega, \mb{\eta'} \rangle \geq 0 \textrm{ for all } \mb{\eta'} \in N_{\pi} \backslash \{ \mb{\eta} \} \}.
\]
It is geometrically obvious that $\mb{\eta}$ is redundant if $\mc{C}(\pi; \mb{\eta}) = \mc{C}(\pi)$ and necessary otherwise. Equivalently, $\eta \in N_{\pi}$ is necessary iff $\mc{C}(\pi)$ is a \textit{proper} subset of $\mc{C}(\pi; \mb{\eta})$. Our strategy is to go through the normal vectors in $N_{\pi}$ and, for each one, try to find a weight vector that is in $\mc{C}(\pi; \mb{\eta})$ but not in $\mc{C}(\pi)$. That this strategy works can be seen by a duality argument, see the remark at the end of this section for more details. It can be used to quickly determine which of the inequalities of the form $\omega(\mb{v}) \geq 0$ are necessary and which are redundant.

\begin{proof}[Proof of Theorem \ref{thm:faces} -- Weights off the path must be positive]\renewcommand{\qedsymbol}{} Suppose $\mb{v} \in P \backslash \{ \pi \}$. Take any $\omega \in \mc{C}(\pi)$ and make the entry at $\mb{v}$ a negative value. This doesn't change the length of $\pi$, and in fact does not increase the length of any other path, so $\pi$ is still maximal under the new vector. This proves that $\mc{C}(\pi)$ is a proper subset of $\mc{C}(\pi;\mb{v})$, and therefore $\omega(\mb{v}) \geq 0$ is a necessary inequality.
\end{proof}

\begin{proof}[Proof of Theorem \ref{thm:faces} -- $\omega(\mb{v}) \geq 0$ for $\mb{v}$ a corner of $\pi$ is redundant]\renewcommand{\qedsymbol}{} If $\mb{v}$ is a corner of $\pi$ then there is another path $\pi'$ such that $\pi\setminus \{\mb{v}\}\subset \pi'$ and $\mb{v}\not\in \pi'$. Define $S=\pi'\setminus \pi$ (note that $S\neq \emptyset$) and $\omega \in \mc{C}(\pi;\mb{v})$. Since $\pi$ is still the longest path for $\omega$ (but possibly $\omega$ is negative at $\mb{v}$), and therefore at least as long as $\pi'$, we get
\[ \omega(\mb{v})\geq \sum_{\mb{v}'\in S} \omega(\mb{v}').\]
But for $\omega\in \mc{C}(\pi;\mb{v})$ we still have that $\omega(\mb{v'}) \geq 0$ for all $\mb{v}'\in S$, hence $\omega(\mb{v}) \geq 0$ also. This implies that $\mc{C}(\pi;\mb{v})=\mc{C}(\pi)$ and the inequality $\omega(\mb{v}) \geq 0$ is redundant.
\end{proof}

\begin{proof}[Proof of Theorem \ref{thm:faces} -- $\omega(\mb{v}) \geq 0$ for $\mb{v}$ on the path but not a corner is necessary]\renewcommand{\qedsymbol}{}
Let $L > 0$, choose $0 < \epsilon < L$, and consider a weight vector $\omega$ defined by
\[
\omega(\mb{u}) = \left\{
                   \begin{array}{ll}
                     -\epsilon, & \mb{u} = \mb{v} \\
                     L, & \mb{u} \in \pi \backslash \{ \mb{v} \} \\
                     0, & \mb{u} \not \in \pi
                   \end{array}
                 \right.
\]
Then $\omega$ is negative at $\mb{v}$, so $\omega\not\in \mc{C}(\pi)$. Now suppose $\pi'$ is longer than $\pi$ for $\omega$. Since the only positive weights are in $\pi \backslash \{ \mb{v} \}$, the only way this is possible is if $\pi \backslash \{ \mb{v} \} \subset \pi'$. But this would imply that $\mb{v}$ is a corner of $\pi$, and this is a contradiction. Since $\omega$ is non-negative on $P \backslash \{ \mb{v} \}$, we see that $\omega \in \mc{C}(\pi; \mb{v})$, making this inequality necessary.
\end{proof}

The necessity and redundancy of the normal vectors of the form $\pi - \pi'$ proof requires a better understanding of the properties of the disorder graph, which leads to the following proposition.

\begin{proposition}\label{prop:disorder}
Suppose $\pi$ and $\pi'$ are two different paths (maximal chains) on $P$. The disorder graph $\Delta(\pi,\pi')$ has the following properties.
\begin{enumerate}[(i)]
    \item $\Delta(\pi,\pi')$ is a bipartite graph, where the two parts are $\pi\setminus \pi'$ and $\pi'\setminus \pi$.
    \item\label{prop:satsubset} If $\Delta(\pi, \pi')$ is connected then $\pi \backslash \pi'$ and $\pi' \backslash \pi$ are saturated subsets of $P$. 
    \item\label{prop:interval} The set of neighbors of any vertex of $\Delta(\pi,\pi')$ is a non-empty saturated subset of $P$ (so it is an interval on the opposite path).
    \item For $\mb{v},\mb{u}\in \pi\setminus \pi'$, denote the set of neighbors in $\pi'$ as $\mc{N}_{\mb{v}}$ and $\mc{N}_{\mb{u}}$ respectively. If $\mb{v}<\mb{u}$, then 
    \[ \min \mc{N}_{\mb{v}}\leq \min \mc{N}_{\mb{u}}\ \ \mbox{and}\ \ \max \mc{N}_{\mb{v}}\leq \max \mc{N}_{\mb{u}}.\]
    The analogous statement holds for $\mb{v},\mb{u}\in \pi'\setminus \pi$.
    \item\label{prop:minimal_out_of_order} The minimal element of $\pi \backslash \pi'$ must be out of order with the minimal element of $\pi' \backslash \pi$, and hence are connected in the disorder graph. The same holds for the maximal elements.
    \item\label{prop:disconnected} Suppose $\Delta(\pi,\pi')$ is not connected. Then there exist $\mb{v}\in \pi$ and $\mb{v}'\in \pi'$ such that \[\{\mb{u}\in\pi \backslash \pi' : \mb{u}<\mb{v}\}\cup \{\mb{u}'\in\pi' \backslash \pi : \mb{u}'<\mb{v}'\} \ \mbox{and}\ \{\mb{u}\in\pi \backslash \pi' : \mb{u}\geq \mb{v}\}\cup \{\mb{u}'\in\pi' \backslash \pi : \mb{u}'\geq \mb{v}'\}\]
    are not connected in $\Delta(\pi, \pi')$, and these four sets are non-empty.
\end{enumerate}
\end{proposition}

\begin{proof}
\mbox{ }
\begin{enumerate}[(i)]
    \item Clearly, all vertices in $\pi$ are ordered, so there is no edge between these vertices, and the same holds for vertices in $\pi'$.
    \item Let $\mb{u}, \mb{w} \in \pi \backslash \pi'$. Suppose $\mb{v} \in \pi$ and $\mb{u} < \mb{v} < \mb{w}$. We want to show that $\Delta(\pi, \pi')$ being connected implies that $\mb{v} \in \pi \backslash \pi'$. Suppose not. Then it must be that $\mb{v} \in \pi \cap \pi'$, from which it follows that every element in $U(\mb{v}) \cap \Delta(\pi, \pi')$ is in order with every element in $L(\mb{v}) \cap \Delta(\pi, \pi')$ (they are in order through $\mb{v}$) hence these two sets are disconnected subsets of $\Delta(\pi, \pi')$. This is a contradiction unless either $U(\mb{v}) \cap \Delta(\pi, \pi') = \emptyset$ or $L(\mb{v}) \cap \Delta(\pi, \pi;) = \emptyset$, which is impossible because $\mb{u} \in L(\mb{v}) \cap \Delta(\pi, \pi')$ and $\mb{w} \in U(\mb{v}) \cap \Delta(\pi, \pi')$.
    \item We will prove a slightly more general statement: for any $\mb{v}\in P\setminus \pi$, the set $\mc{N}_\mb{v}$ of elements in $\pi$ that are out of order with $\mb{v}$ is non-empty and saturated. Suppose $\mc{N}_\mb{v}$ is empty. This means that $\mb{v}$ is in order with every element of $\pi$, contradicting the fact that $\pi$ is a maximal chain. Since $\mc{N}_\mb{v}$ is a non-empty subset of the chain $\pi$, it must have a minimal element denoted by $\mb{u}_0$ and a maximal element denoted by $\mb{u}_1$; it is possible that $\mb{u}_0=\mb{u}_1$. Now suppose $\mb{u}\geq \mb{u}_0$. Then it cannot happen that $\mb{u}<\mb{v}$, since this would imply that $\mb{u}_0<\mb{v}$, and we know that these two vertices are out of order. A completely analogous argument shows that if $\mb{u}\leq \mb{u}_1$, then it cannot happen that $\mb{u}>\mb{v}$. Therefore, if $\mb{u}_0\leq \mb{u}\leq \mb{u}_1$, $\mb{u}$ must be out of order with $\mb{v}$, so $\mc{N}_\mb{v}$ is indeed saturated.
    \item We proceed by contradiction. If $\min \mc{N}_{\mb{u}} < \min \mc{N}_{\mb{v}}$ then there must be a $\mb{u}'\in \mc{N}_\mb{u}$ such that $\mb{u}'<\min \mc{N}_\mb{v}$. The latter means that $\mb{u}'\not\in \mc{N}_\mb{v}$, which implies that $\mb{u}'$  must be in order with $\mb{v}$, i.e. either $\mb{u}' < \mb{v}$ or $\mb{v} < \mb{u}'$. The latter is impossible because it would imply $\mb{v} < \min \mc{N}_\mb{v}$ and we know that these two vertices are out of order. Thus $\mb{u}' < \mb{v}$. But also $\mb{v}<\mb{u}$ by assumption, so therefore $\mb{u'}<\mb{u}$, which contradicts the fact that $\mb{u}'\in \mc{N}_\mb{u}$. Thus $\min \mc{N}_{\mb{u}} < \min \mc{N}_{\mb{v}}$ is impossible, but because $\mc{N}_{\mb{u}}$ and $\mc{N}_{\mb{v}}$ are both subsets of the ordered chain $\pi'$ the only remaining option is that $\min \mc{N}_\mb{v}\leq \min \mc{N}_\mb{u}$. The statements for the max follow from completely similar arguments.
    \item Suppose $\mb{v}$ is the minimal element of $\pi \backslash \pi'$ and $\mb{v}'$ is the minimal element of $\pi' \backslash \pi$. If $\mb{v} <\mb{v}'$ was true, then $\mb{v}$ could be ``inserted'' into the path $\pi'$ to form a longer chain, i.e.
    \[
    (\pi \cap \pi' \cap L(\mb{v})) \cup (\pi' \cap U(\mb{v}))
    \]
    would be a chain in $P$ that contains $\mb{v}$ and $\pi'$. But since $\mb{v} \not \in \pi'$ this contradicts that $\pi'$ is a maximal chain. For the analogous reason we cannot have $\mb{v}'<\mb{v}$. Therefore $\mb{v}$ and $\mb{v}'$ must be out of order.
    \item By part \eqref{prop:minimal_out_of_order} the minimal element of $\pi \backslash \pi'$ is connected to the minimal element of $\pi' \backslash \pi$ in the disorder graph. Let $C$ be the connected component of $\Delta(\pi, \pi')$ containing them both. Then $C \neq \Delta(\pi, \pi')$ by assumption. Thus we may suppose that there exists a minimal $\mb{v} \in \pi \backslash \pi'$ that is not in $C$ (the case of a minimal element of $\pi' \backslash \pi$ not in $C$ is handled similarly). For this $\mb{v}$ it automatically follows that $\{ \mb{u} \in \pi \backslash \pi' : \mb{u} < \mb{v} \}$ is non-empty (it includes the minimal element of $\pi \backslash \pi'$) as is $\{ \mb{u} \in \pi \backslash \pi' : \mb{u} \geq \mb{v} \}$ (it contains $\mb{v}$). By part \eqref{prop:interval}, there is a minimal $\mb{v}' \in \pi' \backslash \pi$ such that $\mb{v}$ and $\mb{v}'$ are out of order (i.e. $\mb{v}'$ is the minimal neighbor of $\mb{v}$ in the disorder graph). Then we must have $\mb{v}' \not \in C$, because if $\mb{v}' \in C$ then $\mb{v}$ being connected to $\mb{v}'$ in the disorder graph would imply $\mb{v} \in C$ also, which contradicts the definition of $\mb{v}$. Thus $\{ \mb{u} \in \pi' \backslash \pi : \mb{u} < \mb{v}' \}$ is non-empty (it contains the minimal element of $\pi' \backslash \pi$ which is in $C$ and therefore different from $\mb{v}'$), as is $\{ \mb{u} \in \pi' \backslash \pi : \mb{u} \geq \mb{v}' \}$. This shows that each of the four sets is non-empty.
    
    Finally, we show that the two sets are not connected in the disorder graph. Let $\mb{u}^* \in \pi \backslash \pi'$ with $\mb{u}^* < \mb{v}$. Then by definition of $\mb{v}$ this means $\mb{u}^* \in C$. Therefore $\mb{u}^*$ and $\mb{v}'$ cannot be connected in the disorder graph (else it would imply $\mb{v} \in C)$ so therefore $\mb{u}^* < \mb{v}'$. Thus $\mb{u}^*$ cannot be connected in the disorder graph to $\{ \mb{u}' \in \pi \backslash \pi' : \mb{u}' \geq \mb{v} \}$. Further $\mb{u}^*$ already cannot be connected to $\{ \mb{u} \in \pi \backslash \pi' : \mb{u} \geq \mb{v} \}$. Thus $\{ \mb{u} \in \pi \backslash \pi' : \mb{u} < \mb{v} \}$ is not connected to $\{ \mb{u} \in \pi \backslash \pi' : \mb{u} \geq \mb{v} \} \cup \{ \mb{u}' \in \pi' \backslash \pi : \mb{u}' \geq \mb{v}' \}$, and by an analogous argument the latter set is also not connected to $\{ \mb{u} \in \pi' \backslash \pi : \mb{u} < \mb{v}' \}$. This completes the proof.
    
    %Let $C$ be the connected component of $\Delta(\pi,\pi')$ containing the minimal elements of $\pi$ and $\pi'$. Since $C\neq \Delta(\pi,\pi')$, there exists a minimal element $\mb{v}\in \pi$ not in $C$, or there exists a minimal element $\mb{v}'\in\pi'$ not in $C$. Suppose we have such a $\mb{v}\in \pi$ (the case where we have $\mb{v}'\in\pi'$ follows in the same way). Because of \ref{prop:interval}, we know that there exists a minimal element $\mb{v}'\in\pi'$ such that $\mb{v}$ is out of order with $\mb{v}'$. Note that $\mb{v}'\not\in C$! This already shows that the four sets $\{\mb{u}\in\pi: \mb{u}<\mb{v}\}, \{\mb{u}'\in\pi': \mb{u}'<\mb{v}'\}, \{\mb{u}\in\pi: \mb{u}\geq \mb{v}\}$ and $\{\mb{u}'\in\pi': \mb{u}'\geq \mb{v}'\}$ are non-empty. Now choose $\mb{u}<\mb{v}$. First of all, this means that $\mb{u}\in C$. Therefore, $\mb{u}$ and $\mb{v}'$ cannot be connected in $\Delta(\pi,\pi')$, and this implies that $\mb{u}<\mb{v}'$. So $\mb{u}$ cannot be connected in $\Delta(\pi,\pi')$ to $\{\mb{u}'\in\pi': \mb{u}'\geq \mb{v}'\}$. When we choose $\mb{u}'<\mb{v}'$, we also know that $\mb{u}'$ is not connected to $\mb{v}$ (since $\mb{v}'$ was the minimal element not connected to $\mb{v}$), so $\mb{u}'<\mb{v}$, and the conclusion is again that $\mb{u}'$ cannot be connected to $\{\mb{u}\in\pi: \mb{u}\geq \mb{v}\}$. This proves our statement.

\end{enumerate}
\end{proof}

Now we return to the redundancy and necessity of inequalities of the form $\langle \pi - \pi', \omega \rangle \geq 0$.

\begin{proof}[Proof of Theorem \ref{thm:faces} -- If $\Delta(\pi, \pi')$ is disconnected then $\pi - \pi'$ is redundant.]\renewcommand{\qedsymbol}{}
By part \eqref{prop:disconnected} of Proposition \ref{prop:disorder}, there exists $\mb{v}\in \pi$ and $\mb{v}'\in \pi'$ such that $\mb{v}$ is above every vertex $\mb{u}'<\mb{v}'$, and $\mb{v}'$ is above every vertex $\mb{u}<\mb{v}$, and these two sets are non-empty. Let $\pi_1$ be the path which follows $\pi'$ until just before $\mb{v}'$ and then switches to $\mb{v}$ and follows $\pi$ afterwards. Then $\langle \pi - \pi_1, \omega \rangle \geq 0$ on $\mc{C}(\pi; \pi-\pi')$. Similarly, let $\pi_2$ be the path which follows $\pi$ until just before $\mb{v}$, then switches to $\mb{v}'$ and follows $\pi'$ afterwards. Then $\langle \pi - \pi_2 , \omega \rangle \geq 0$ on $\mc{C}(\pi;\pi-\pi')$. But since $\pi - \pi_1 + \pi - \pi_2 = \pi - \pi'$, this implies that $\langle \pi - \pi' , \omega \rangle \geq 0$ automatically on $\mc{C}(\pi;\pi-\pi')$, which means that $\pi - \pi'$ is redundant (or $\mc{C}(\pi;\pi-\pi')=\mc{C}(\pi)$). 
\end{proof}

\begin{proof}[Proof of Theorem \ref{thm:faces} -- If $\Delta(\pi, \pi')$ is connected then $\pi - \pi'$ is necessary.]%\renewcommand{\qedsymbol}{}
We need to show that the cone $\mc{C}(\pi; \pi - \pi')$ is strictly larger than $\mc{C}(\pi)$, meaning that there is a weight $\omega \in \R_+^P$ under which $\pi'$ is the unique longest path and $\pi$ is the second longest path, although $\pi$ may be tied for second longest with several other paths. We will give an explicit such weight vector. To do so we write $\pi \backslash \pi'$ as $\mb{u}_0 \lessdot \mb{u}_1 \lessdot \ldots \lessdot \mb{u}_{\ell}$, recalling that $\pi \backslash \pi'$ is a saturated subset of $P$ because the disorder graph is connected (Proposition \ref{prop:disorder}, part \ref{prop:satsubset}). Recall that we denote the upper set of $\mb{u}$ by $U(\mb{u})$. The claimed weight vector $\omega$ is 
\begin{itemize}
    \item $\omega(\mb{v}) = 1$ for $\mb{v} \in \pi'$,
    \item $\omega(\mb{u}_0) = |\pi' \backslash \pi| - |(\pi' \backslash \pi) \cap U(\mb{u}_0)| - 1$, and
    \item $\omega(\mb{u}_i) = |(\pi' \backslash \pi) \cap U(\mb{u}_{i-1})| - |(\pi' \backslash \pi) \cap U(\mb{u}_i)|$ for $i=1,\ldots,\ell$,
    \item $\omega(\mb{v}) = 0$ otherwise.
\end{itemize}
Note that all weights are non-negative, since $U(\mb{u}_{i})\subset U(\mb{u}_{i-1})$ and there exists at least one vertex in $\pi'\setminus \pi$ that is not an element of $U(\mb{u}_0)$ (the lowest element of $\pi'\setminus \pi$ is out of order with $\mb{u}_0$, by part \eqref{prop:minimal_out_of_order} of Proposition \ref{prop:disorder}). Under this weight vector the length of $\pi'$ is $\langle \pi', \omega \rangle = |\pi'|$, since there is weight one on each element of $\pi'$. Furthermore, the length of $\pi$ under $\omega$ is 
\begin{align*}
\langle \pi, \omega \rangle &= \langle \pi \cap \pi', \omega \rangle + \langle \pi \backslash \pi', \omega \rangle \\
&= |\pi \cap \pi'| + \omega(\mb{u}_0) + \omega(\mb{u}_1) + \ldots + \omega(\mb{u}_{\ell}) \\
&= |\pi \cap \pi'| + |\pi' \backslash \pi| - 1 - |(\pi' \backslash \pi) \cap U(\mb{u}_{\ell})| \\
&= |\pi'| - 1 \\
&= \langle \pi', \omega \rangle - 1.
\end{align*}
The fourth equality follows because $\mb{u}_{\ell}$, the highest element of $\pi \backslash \pi'$, is out of order with the highest element of $\pi' \backslash \pi$ and hence \textit{no} element of $\pi' \backslash \pi$ can be above $\mb{u}_\ell$. Thus $|(\pi' \backslash \pi) \cap U(\mb{u}_{\ell})| = 0$. Now since $\omega$ has only integer entries, any paths with distinct weights must have lengths that differ by at least one, hence $\langle \pi, \omega \rangle = \langle \pi', \omega \rangle - 1$ shows that $\pi$ is necessarily a second longest path. Thus $\omega$ satisfies the required conditions so long as there is no other path (distinct from $\pi$ and $\pi'$) with the same length as $\pi'$ under $\omega$. Since the non-zero weights in $\omega$ are only distributed along $\pi \cup \pi'$, any such path would have to switch between $\pi$ and $\pi'$, while also containing the common part $\pi \cap \pi'$.

We start by proving that switching from $\pi$ to $\pi'$ will not increase the length of a path. To do this we show that wherever we switch along $\pi$, the remaining weight available along $\pi'$ is the same as the remaining weight available along $\pi$. Suppose that $\mb{u}_i$ for $i\geq 1$ is the first weight along $\pi$ that we do not use, meaning that $\mb{u}_{i-1}$ is the last element along $\pi$ that we still go through. Then the weight available along $\pi'$ by making this switch is $|U(\mb{u}_{i-1})\cap (\pi'\backslash \pi)|$. On the other hand, had we stayed on $\pi$ the remaining weight we would have picked up on $\pi \backslash \pi'$ is
\[ 
\sum_{j=i}^\ell |(\pi' \backslash \pi) \cap U(\mb{u}_{j-1})| - |(\pi' \backslash \pi) \cap U(\mb{u}_j)| = |(\pi' \backslash \pi) \cap U(\mb{u}_{i-1})|.
\]
Again the last equality uses that $|(\pi' \backslash \pi) \cap U(\mb{u}_{\ell})| = 0$. This proves that making a switch from $\pi$ to $\pi'$ and staying on $\pi'$ is never profitable. 

Now consider a switch from $\pi'$ to $\pi$. Let $\mb{v}'$ be the element of $\pi' \backslash \pi$ that we switch from and $\mb{u}_i$ be the first element of $\pi \backslash \pi'$ that it is possible to switch to from $\mb{v}'$. Then it must be that $\mb{v}' < \mb{u}_i$. Let $\mb{w}'$ be the unique element in $\pi'$ that covers $\mb{v}'$ (so it is the next point on $\pi'$ after $\mb{v}'$). It cannot be that $\mb{u_i} < \mb{w}'$, since otherwise the path $\pi'$ could be extended by going from $\mb{v}$ to $\mb{u_i}$ to $\mb{w}'$; this would contradict that $\pi'$ is a maximal chain. Thus $\mb{w}' \in \pi' \backslash \pi$ also. Now suppose $\mb{u}_{i-1} < \mb{w}'$. Then $\{ \mb{u}_0, \ldots, \mb{u}_{i-1} \}$ is not connected in $\Delta(\pi, \pi')$ to the set $\{ \mb{u}' \in \pi' : \mb{u}' \geq \mb{w}' \}$, while $\{ \mb{u}_i, \ldots, \mb{u}_l \}$ and $\{ \mb{u}' \in \pi' : \mb{u}' < \mb{w}' \}$ are not connected either: this would imply that $\Delta(\pi, \pi')$ is not connected. Thus $\mb{u}_{i-1} < \mb{w}$ is impossible, and therefore $\mb{w}' \not \in U(\mb{u}_{i-1})$. This implies that
\[
|(\pi' \backslash \pi) \cap U(\mb{w}')| > |(\pi' \backslash \pi) \cap U(\mb{u}_{i-1})|.
\]
But the left hand side is how much weight would be picked up along $\pi' \backslash \pi$ by staying on $\pi'$ after $\mb{v}'$, while the right hand side is the weight that would be picked up along $\pi \backslash \pi'$ by switching from $\mb{v}'$ to $\mb{u}_i$, since
\[ 
|(\pi' \backslash \pi) \cap U(\mb{u}_{i-1})| = \sum_{j=i}^\ell |(\pi' \backslash \pi) \cap U(\mb{u}_{j-1})| - |(\pi' \backslash \pi) \cap U(\mb{u}_j)| = \sum_{j=i}^{\ell} \omega(\mb{u}_j). 
\]
Thus it is more profitable to stay on $\pi'$ than to switch to $\pi$.

Together, these two facts imply that $\pi'$ is the unique longest path, and therefore that $\omega \in \mc{C}(\pi;\pi-\pi')\setminus \mc{C}(\pi)$.
\end{proof}

\begin{remark}
 A more standard way of proving Theorem \ref{thm:faces} would be to use the dual cone
\[
 \mc{C}(\pi)^* = \{ \xi \in \R^P : \langle \xi, \omega \rangle \geq 0 \textrm{ for all } \omega \in \mc{C}(\pi)  \},
\]
which defines the set of half-spaces containing $\mc{C}(\pi)$ (through their normal vectors). The extreme rays of $\mc{C}(\pi)^*$ are exactly the normals to the facets of $\mc{C}(\pi)$, and arguments similar to those in Section \ref{sec:extreme_rays} can be used to verify which vectors are extreme rays of the dual cone. We chose the exposition above since it is more in the spirit of the last passage model, but the duality argument is also useful in its own right. For example, it makes clear the assertion that $\eta$ is necessary iff $\mc{C}(\pi) \subsetneq \mc{C}(\pi; \eta)$, since on the dual side it corresponds to the statement that a cone becomes smaller when an extreme ray is removed from it. Moreover, the duality argument in this case is made much simpler by the nature of the extreme rays to $\mc{C}(\pi)$, as proved in Theorem \ref{thm:extreme_rays}. Indeed, using that theorem the dual cone can be rewritten as
\begin{align}\label{eqn:dual_C_pi}
 \mc{C}(\pi)^* = \{ \xi \in \R^P : \langle \xi, a \rangle \geq 0 \textrm{ for every antichain } a \textrm{ with } \supp(a) \cap \supp(\pi) \neq \emptyset \}.
\end{align}
The only candidates for extreme rays of the dual cone are vectors of the form $\xi = \delta_{\mb{v}}$ or $\xi = \pi - \pi'$, and in that case it is clear that $\langle \xi, a \rangle$ is always either $0$ or $1$, for any antichain $a$ that intersects $\pi$ exactly once. Using this fact, one can show that a vector $\xi$ of the type either $\delta_{\mb{v}}$ or $\pi - \pi'$ is an extreme ray of $\mc{C}(\pi)^*$ if and only if
\begin{align}\label{eqn:dual_extreme_ray_condition}
 \bigcap_{\substack{a \in \mcr{ER}(\pi) :\\ \langle \xi, a \rangle = 0}} \!\! \Span(a)^{\perp} = \Span(\xi),
\end{align}
where $\mcr{ER}(\pi)$ is the set of extreme rays of $\mc{C}(\pi)$. This simple fact can be used to make a relatively quick determination of the extreme rays of the dual cone. Specifically, for the $\xi$ of the form $\pi - \pi'$ the relevant antichains are those which are supported on one vertex of $\pi \backslash \pi'$ and one vertex of $\pi' \backslash \pi$, with the two vertices being out of order to maintain the antichain condition. This forces that any vector in the left hand side of \eqref{eqn:dual_extreme_ray_condition} must have entries with opposite values at each such pair of vertices, and thus on each connected component of $\Delta(\pi, \pi')$ the entries on each partite set within the component must be the same, and the negative of the common entry on the other partite set.  
\end{remark}

\section{Two-Dimensional Faces of Maximal Cones \label{sec:two_dim_faces}}

Fix a path $\pi \in \Pi_P$. In this section we will describe when two extreme rays $a_1, a_2 \in \mcr{ER}(\pi)$ form a two-dimensional face of $\mc{C}(\pi)$. There are several equivalent definitions of what this means, and we will use one that has a description in terms of the last passage problem. Clearly since $a_1$ and $a_2$ are in $\mc{C}(\pi)$ so too is their sum $a_1 + a_2$. Then $a_1$ and $a_2$ form a two-dimensional face precisely when there are \textit{only} two linearly independent directions from which one can perturb away from $a_1 + a_2$ and remain in the cone $\mc{C}(\pi)$. Formally this means
\begin{align}\label{eqn:two_dim_face}
a_1, a_2 \textrm{ form a two-dimensional face of } \mc{C}(\pi) \iff \mc{D}_{\pi}(a_1 + a_2) = \Span \{ a_1, a_2 \}.
\end{align}
We will assume throughout that $a_1$ and $a_2$ are distinct, otherwise they clearly don't form a two-dimensional face. Since they are extreme rays this means there must be at least one vertex where $a_1$ or $a_2$ takes on the value one and the other is zero. It \textit{is} possible that there are vertices at which both $a_1$ and $a_2$ take the value one, but excluding the case $a_1 = a_2$ means that $\supp(a_1)$ and $\supp(a_2)$ are necessarily distinct. With this in mind we state the following result about edges. It relies on an object which we define next, called the \textbf{order graph}.

\begin{definition}\label{defn:order_graph}
For two extreme rays $a_1$ and $a_2$ of a maximal cone $\mc{C}(\pi)$, their \textit{order graph} is the bipartite graph $\mc{G}(a_1, a_2)$ with one part being the vertices in $\supp(a_1) \backslash \supp(a_2)$, the other part being the vertices in $\supp(a_2) \backslash \supp(a_1)$, and an edge between two vertices if they are in order.
\end{definition}

With this definition in hand the result is:

\begin{theorem}\label{thm:edges}
Two extreme vectors $a_1, a_2$ of $\mc{C}(\pi)$ form a two-dimensional face in the cone $\mc{C}(\pi)$ iff their order graph $\mc{G}(a_1, a_2)$ is connected.
\end{theorem}

Observe that the order graph and the theorem don't make any reference to the vertices in $\supp(a_1) \cap \supp(a_2)$. Those vertices would only make trivial changes to the order graph. Since they appear in both extreme rays the natural choice would be to include the vertex in both partite sets with an edge between them. However, such vertices cannot be in order with any of the other vertices from either $a_1$ or $a_2$ since they are extreme rays. Therefore they would only appear in the graph as isolated components, and we will see in the proof that their presence would only make trivial changes to the statement of the theorem.

Also note that because $a_1 \neq a_2$ the order graph is always non-empty, but it is possible that one of the parts $\supp(a_1) \backslash \supp(a_2)$ or $\supp(a_2) \backslash \supp(a_1)$ is empty. In that case the order graph is connected iff the non-empty part consists of exactly one element. Thus it is always possible to build two-dimensional faces of $\mc{C}(\pi)$ by picking an arbitrary $a_1 \in \mcr{ER}(\pi)$ and then forming $a_2$ by adding one element of $P$ that is out of order with $\supp(a_1)$, if such an element exists. Adding more than one out of order element will \textit{not} form a two-dimensional face. As the theorem shows, however, not all two-dimensional faces come about from this type of construction.

In the rest of this section we prove Theorem \ref{thm:edges}. We begin with simplifications of \eqref{eqn:two_dim_face}.

\begin{lemma}\label{lem:ERD}
\begin{align*}
\mc{D}_{\pi}(\omega) = \operatorname{Span} \left \{ a \in \mcr{ER}(\pi) : \, \exists \, \epsilon > 0 \textrm{ s.t. } \omega \pm \epsilon a \in \mc{C}(\pi) \right \}.
\end{align*}
\end{lemma}
\begin{proof}
Define the set 
\[ E = \left \{ a \in \mcr{ER}(\pi) : \, \exists \, \epsilon > 0 \textrm{ s.t. } \omega \pm \epsilon a \in \mc{C}(\pi) \right \}.\]
Clearly, $\operatorname{Span}(E)\subset \mc{D}_{\pi}(\omega)$. Now suppose $\omega\in \mc{C}(\pi)$ and $\omega \pm \epsilon \sigma \in \mc{C}(\pi)$. There exist subsets $F,G \subset \mcr{ER}(\pi)$ such that $\omega +\epsilon \sigma$ and $\omega -\epsilon \sigma$ are strictly positive linear combinations of all extreme rays in $F$ and $G$ respectively. But then $\omega$, as the mean of these two vectors, will be a strictly positive combination of all elements of $F\cup G$. Therefore, $F\cup G\subset E$, so $\sigma \in \operatorname{Span}(E)$.
\end{proof}

\begin{remark}
Suppose $a_1, a_2\in \mcr{ER}(\pi)$. Clearly $a_1$ and $a_2$ are always in $\mc{C}(\pi)$. Therefore, using Lemma \ref{lem:ERD}, $a_1$ and $a_2$ forming a face is equivalent to the fact that for all $a\in \mcr{ER}(\pi)$
\begin{align*}
a_1 + a_2 \pm \epsilon a \in \mc{C}(\pi) \iff a = a_1 \textrm{ or } a_2,
\end{align*}
or, more succinctly, that $\mcr{ER}(\pi) \cap D_{\pi}(a_1 + a_2) = \{ a_1, a_2 \}$.
\end{remark}

Using the last remark and the properties of extreme rays it is possible to greatly reduce the number of extreme rays $a \in \mcr{ER}(\pi)$ that are possibly in $\mc{D}_{\pi}(a_1 + a_2)$.

\begin{lemma}\label{lem:A}
Let $a$ be an extreme ray of $\mc{C}(\pi)$. Then $a \in \mc{D}_{\pi}(a_1 + a_2)$ iff $a$ is zero at all vertices where both $a_1$ and $a_2$ are (i.e. $\supp(a) \subset \supp(a_1) \cup \supp(a_2) = \supp(a_1 + a_2)$) and any path that is maximal for $a_1 + a_2$, is also maximal for $a$.
\end{lemma}

\begin{proof}
First suppose $a \in \mc{D}_{\pi}(a_1 + a_2)$. Then the vector $a_1 + a_2 - \epsilon a$ must be non-negative at every vertex. Thus if $a_1$ and $a_2$ are both zero at some vertex then so must be $a$.

For the second part, suppose that $\pi'$ is a maximal path for $a_1+a_2$. Note that $\pi$ is a maximal path for all the vectors $a_1+a_2, a$ and $a_1+a_2\pm \epsilon a$. So using that $\ip{a_1+a_2}{\pi-\pi'}=0$ and $\ip{a_1+a_2\pm\epsilon a}{\pi-\pi'}\geq 0$, we see that $\ip{a}{\pi'}=\ip{a}{\pi}$. This shows that $\pi'$ is also a maximal pah for $a$.

For the reversed statement, suppose that $\supp(a) \subset \supp(a_1) \cup \supp(a_2)$ and that all paths maximal for $a_1+a_2$ are also maximal for $a$. Clearly, for $\epsilon$ small enough, $a_1+a_2-\epsilon a$ is still non-negative at every vertex. Now consider a path $\pi'$. If $\ip{a_1+a_2}{\pi}>\ip{a_1+a_2}{\pi'}$, then again for $\epsilon$ small enough and for all such $\pi'$ we would have
\[ \ip{a_1+a_2\pm\epsilon a}{\pi-\pi'}>0.\]
Now suppose that $\ip{a_1+a_2}{\pi'}=\ip{a_1+a_2}{\pi}$, so $\pi'$ is a maximal path for $a_1+a_2$. Then it is also maximal for $a$, so we get
\[ \ip{a_1+a_2\pm\epsilon a}{\pi-\pi'}=0.\]
This shows that indeed $a\in \mc{D}_\pi(a_1+a_2)$.
\end{proof}

In geometric terms the previous lemma can be recast as saying that $a \in \mcr{ER}(\pi) \cap \mc{D}_{\pi}(a_1 + a_2)$ iff $\supp(a) \subset \supp(a_1 + a_2)$ and
\[
a \in \!\!\!\!\!\!\!\!\! \bigcap_{\substack{\pi' \in \Pi_P : \\ \ip{a_1 + a_2}{\pi - \pi'} = 0}} \!\!\!\!\!\!\!\!\! \Span(\pi - \pi')^{\perp}.
\]
However, it is simpler to recast it in terms of the order graph.

\begin{lemma}\label{lem:B}
Let $a$ be an extreme ray of $\mc{C}(\pi)$. Then $a \in \mc{D}_{\pi}(a_1 + a_2)$ iff
\begin{enumerate}
    \item $\supp(a) \subset \supp(a_1) \cup \supp(a_2)$.
    \item For every edge $(\mb{v}, \mb{w}) \in \mc{G}(a_1, a_2)$ either $a(\mb{v}) = 1$ or $a(\mb{w}) = 1$ (but not both).
    \item If $\mb{v} \in \supp(a_1) \cap \supp(a_2)$ then $a(\mb{v}) = 1$.
\end{enumerate}
\end{lemma}

\begin{proof}
Suppose $a\in \mc{D}_{\pi}(a_1 + a_2)$. The first condition follows immediately from Lemma \ref{lem:A}. Now suppose $(\mb{v}, \mb{w}) \in \mc{G}(a_1, a_2)$. This means that $\mb{v}$ and $\mb{w}$ are in order, and hence there exists a path $\pi'$ passing through both $\mb{v}$ and $\mb{w}$. Therefore $\pi'$ is maximal under $a_1 + a_2$ (a path can pick up at most one vertex of an extreme ray, since all its vertices are out of order). But then by Lemma \ref{lem:A} $\pi'$ is also maximal under any $a \in \mcr{ER}(\pi) \cap \mc{D}_{\pi}(a_1 + a_2)$. Since $\pi'$ can't pass through any other vertices from $\supp(a_1) \cup \supp(a_2)$ other than $\mb{v}$ or $\mb{w}$, $a$ must be non-zero at either $\mb{v}$ or $\mb{w}$ in order for $\pi'$ to be maximal under $a$ (here we use Condition 1). However $a$ can't be non-zero at both, otherwise $a$ would be not be extreme.

For $\mb{v} \in \supp(a_1) \cap \supp(a_2)$ any path that passes through $\mb{v}$ is longest under $a_1 + a_2$, and hence must also be longest under $a$. But this implies that $a(\mb{v}) = 1$.

Now we wish to prove the reverse implication. Suppose $\pi'$ is maximal for $a_1+a_2$. Since $\pi$ is maximal for both $a_1$ and $a_2$, either there exists an edge in $\mc{G}(a_1, a_2)$ or $\supp(a_1)\cap \supp(a_2) \neq \emptyset$. In both cases we would have that $\ip{\pi'}{a_1+a_2}=2$ (otherwise $\pi'$ would not be maximal). If $\pi'$ picks up two vertices $\supp(a_1) \cup \supp(a_2)$, then Condition 2 implies that $\ip{\pi'}{a}=1$, so $\pi'$ is maximal for $a$. If $\pi'$ picks up a vertex $\mb{v}$ in $\supp(a_1)\cap \supp(a_2)$, then Condition 3 implies that $a(\mb{v})=1$, so also in that case $\pi'$ is maximal for $a$. Now Lemma \ref{lem:A} implies that $a\in \mc{D}_{\pi}(a_1 + a_2)$.
\end{proof}

This allows us to finish the proof of the theorem.

\begin{proof}[Proof of Theorem \ref{thm:edges}]
First assume that $a_1$ and $a_2$ form a two-dimensional face. Then $a \in \mc{D}_{\pi}(a_1 + a_2)\cap \mcr{ER}(\pi)$ implies that $a = a_1$ or $a = a_2$. Suppose $\mc{C}\subset \mc{G}(a_1,a_2)$ is a connected component, strictly smaller than $\mc{G}(a_1,a_2)$. Define $a\in \R^P$ in the following way: $a(\mb{v})=1$ for all $\mb{v}\in \supp(a_1)\cap \supp(a_2)$, all $\mb{v}\in \mc{C}\cap \supp(a_1)$ and all $\mb{v}\in (\mc{G}(a_1,a_2)\setminus \mc{C})\cap \supp(a_2)$. For all other $\mb{v}\in P$ $a(\mb{v})=0$. Note that $a\in \mcr{ER}(\pi)$: two vertices in $\supp(a)$ cannot be in order, since all elements of $\mc{C}$ are out of order with all elements of $\mc{G}(a_1,a_2)\setminus \mc{C}$. Furthermore, define $\mb{v}_i$ (for $i=1,2$) as the unique element of $\supp(\pi)\cap \supp(a_i)$. If $\mb{v}_1=\mb{v}_2$, then $\mb{v}_1\in \supp(a_1)\cap\supp(a_2)$, so $a(\mb{v}_1)=1$. If $\mb{v}_1\neq \mb{v}_2$, then they are in order, so they are either both in $\mc{C}$ (so $a(\mb{v_1})=1$), or both in $\mc{G}(a_1,a_2)\setminus \mc{C}$ (so $a(\mb{v}_2)=1$). In all of these cases we have that $\ip{\pi}{a}=1$. It is also clear that $a$ satisfies the three conditions of Lemma \ref{lem:B}, so we conclude that $a\in \mc{D}_\pi(a_1+a_2)\cap \mcr{ER}(\pi)$. However, $a$ cannot be equal to $a_1$ or $a_2$: there exists at least one vertex $\mb{v}\in \mc{C}$. If $\mb{v}\in \supp(a_1)\setminus \supp(a_2)$, then $a(\mb{v})=1$. If $\mb{v}\in \supp(a_2)\setminus \supp(a_1)$, then $a(\mb{v})=0$. In both cases we have $a\neq a_2$. Since there exists also at least one vertex in $\mc{G}(a_1,a_2)\setminus \mc{C}$, we also see that $a\neq a_1$. We now contradict the assumption we started with, so $\mc{G}(a_1,a_2)$ must be connected.

For the reverse statement, we use a similar approach. Suppose $a\in \mc{D}_\pi(a_1+a_2)\cap \mcr{ER}(\pi)$. Then $a$ must satisfy the three conditions of Lemma \ref{lem:B}. There also must be at least one vertex $\mb{v}\in \mc{G}(a_1,a_2)$. Suppose $\mb{v}\in \supp(a_1)\setminus \supp(a_2)$ (the other case follows the same arguments). Suppose $a(\mb{v})=1$. Since $\mc{G}(a_1,a_2)$ is a connected bipartite graph and $a$ satisfies Condition 2 of Lemma \ref{lem:B}, it follows that $a$ must be $0$ on $\supp(a_2)\setminus \supp(a_1)$ and $1$ on $\supp(a_1)\setminus \supp(a_2)$, and therefore (together with Condition 3 of Lemma \ref{lem:B}) $a=a_1$. If, on the other hand, $a(\mb{v})=0$, then by the same reasoning we would have that $a=a_2$. This proves that $a_1$ and $a_2$ form a two-dimensional face.
\end{proof}

\begin{remark}
We note that the definition of the order graph seems to somehow be ``dual'' to the definition of the disorder graph, especially one if considers paths and antichains as being in duality. However the order graph describes dimension 2 faces of the maximal cones, while the disorder graph describes the codimension 1 facets. We are unsure of the exact reason for this discrepancy.
\end{remark}

\section{A Simplicial Decomposition of Maximal Sets \label{sec:simplicial_decomp}}

Recall that a polyhedral cone is simplicial if its extreme rays form a basis for the ambient space. The major consequence of this fact is that every point in the cone can be uniquely written as a positive, linear combination of the extreme rays, whereas the uniqueness fails for a non-simplicial cone. Put another way, a simplicial cone is the image of a one-to-one linear transformation of the positive orthant of the ambient Euclidean space (the map that sends the standard basis vectors to the extreme rays of the cone), and as a result integrals over the simplicial cone can be transformed into integrals over the positive orthant. The results of Section \ref{sec:extreme_rays} show that the number of extreme rays of a maximal cone $\mc{C}(\pi)$ is typically much larger than the dimension $|P|$ of the Euclidean space that the problem is embedded into, and hence the cones are far from simplicial. However, computations can be made tractable by partitioning the polyhedral cone into a disjoint union of simplices (disjoint up to measure zero boundary intersections), and general theory ensures that such a partition always exists. In fact, it is always possible to find a decomposition such that the extreme rays of every simplical cone are also extreme rays of the original cone. In this section we describe a general scheme for finding such a decomposition for the last passage model. Although we found this scheme independently, it already appears in \cite{stanley:two_poset_polytopes}.

The key is to consider the set of upper sets of the poset, sometimes also called the \textit{order ideals}. Recall that $U \subset P$ is an upper set if $\mb{v} \in U$ and $\mb{v} \leq \mb{u}$ implies $\mb{u} \in U$. Further recall that $\bd U$ is the set of minimal elements of $U$, which clearly forms an antichain, and this establishes a bijection between upper sets and antichains. In the context of the last passage model this bijection is very natural: if an antichain is the weight vector then its corresponding upper set is the vector of passage times, which encodes the maximal length up to each given vertex.

The set of all upper sets of $P$, ordered by inclusion, itself forms a poset called $J(P)$. It is well known and straightforward to see that $J(P)$ is in fact a \textit{distributive lattice}, meaning that any two elements of $J(P)$ have a unique least upper bound and greatest lower bound that are, in this case, given by the union and intersection of the elements, respectively. The unique minimal element of $J(P)$ is the empty set and the unique maximal element is $P$ itself. Furthermore, $J(P)$ is graded of rank $|P|$ and in this case the rank function $\rho(U)$ of $U \in J(P)$ is simply the cardinality of $U$. In particular the maximal chains of $J(P)$ all contain $|P| + 1$ elements and can all be written in the form $U_0 \lessdot U_1 \lessdot \ldots \lessdot U_{|P|}$ where each lower set $U_j$ contains exactly $j$ vertices in $P$. The maximal chains of $J(P)$ can be used to produce a simplicial decomposition of the maximal sets $\mc{C}(\pi)$. The construction is best explained through an object called the \textit{order cone} of $P$, which is a mild generalization of the order polytope of Stanley \cite{stanley:two_poset_polytopes}.

\begin{definition}
Define the \textit{order cone} $\mcr{OC}(P)$ of $P$ to be the subset of vectors in $\R_+^P$ that obey the ordering of the poset, i.e.
\[
 \mcr{OC}(P) := \left \{ \eta \in \R_+^P : \mb{v} \leq \mb{w} \implies \eta(\mb{v}) \leq \eta(\mb{w})  \right \}.
\]
It is straightforward to verifty that $\mcr{OC}(P)$ is a polyhedral cone, but also that it is a proper subset of $\R_+^P$. 
\end{definition}

\begin{theorem}\label{thm:simplicial_cones}
Let $U_0 \lessdot U_1 \lessdot \ldots \lessdot U_{|P|}$ be a maximal chain in $J(P)$. The conical combinations of the antichains $\bd U_1, \bd U_2, \ldots, \bd U_{|P|}$, embedded as vectors in $\R_+^P$, form a simplicial cone in $\R_+^P$. Moreover the set of all such simplicial cones forms a partition of $\R_+^P$ (up to measure zero boundaries).
\end{theorem}

\begin{proof}
$J(P)$ forms a simplicial decomposition for $\mcr{OC}(P)$ in the following way: for any $\eta \in \mcr{OC}(P)$ and for $k = 1, 2, \ldots, |P|$ let $\mb{v}_k$ be the vertex at which $\eta$ achieves its $k^{th}$ smallest value. These $\mb{v}_k$ are well-defined so long as $\eta$ doesn't take on the same value at multiple vertices, which is Lebesgue almost all of $\mcr{OC}(P)$. Let $U^*_k = \{ \mb{v}_i : i \geq k \}$ for $k=1, \ldots, |P|$ and $U^*_0 = \emptyset$. Then each $U^*_k$ is an upper set of $P$, due to the $\eta$ obeying the order relation, and since $U^*_k \subset U^*_{k+1}$ with $U^*_{k+1} \backslash U^*_k = \{ \mb{v}_k \}$ it follows that $U^*_0 \lessdot U^*_1 \lessdot \ldots U^*_{|P|}$ is a maximal chain in $J(P)$. Thus each $\eta \in \mcr{OC}(P)$ produces a maximal chain in $J(P)$, and furthermore the set of vectors $\eta$ which produce any particular maximal chain forms a simplicial cone in $\mcr{OC}(P)$. It is in fact a canonical simplicial cone, since it is determined by an ordering of the $\eta$ variables, and as such the indicator functions of the upper sets $U^*_1, U^*_2, \ldots, U^*_{|P|}$ are its extreme rays. Since each vector in $\mcr{OC}(P)$ uniquely determines one of these simplicial cones (Lebesgue almost surely), it follows that we have a simplicial decomposition. Moreover this construction is reversible: every maximal chain $\{U_k\}_{k=0,\ldots,|P|}$ in $J(P)$ determines a sequence $\mb{v}_{k+1} := U_{k+1} \backslash U_k$ and the simplicial cone $\{ \eta \in \R_+^P : 0 \leq \eta(\mb{v}_1) \leq \eta(\mb{v}_2) \leq \ldots \leq \eta(\mb{v}_{|P|}) \}$ is in $\mcr{OC}(P)$.

To complete the proof we now associate each simplicial cone in $\mcr{OC}(P)$ to one in $\R_+^P$. Use the linear mapping that takes $U_i \to \bd U_i$ for $i=1,\ldots, |P|$, which is invertible as a transformation of $\R^P$ due to the bijection between upper sets and antichains. Hence it takes the simplex with extreme rays $U_i$ to a simplex with extreme rays $\bd U_i$, which is clearly in $\R_+^P$ since the $\bd U_i$ are. Note that it is a \textit{different} linear map for each maximal chain $\{U_k\}_{k=0,\ldots,|P|}$, and hence a different linear map applied to each simplex defined by the maximal chain. Each mapping is simply the map from the passage time vector back to the weight vector. Since each weight vector uniquely determines an ordering of the passage times (Lebesgue almost surely) one sees that disjoint simplices in $\mcr{OC}(P)$ are mapped to disjoint simplices in $\R_+^P$ (up to measure zero boundaries), and that in fact all of $\R_+^P$ is covered by these mappings from $\mcr{OC}(P)$ to $\R_+^P$. This completes the proof.
\end{proof}

%\begin{proof}
%Since we are taking the conical combination of $|P|$ vectors in a $|P|$-dimensional space it is enough to check that they are linearly independent. Since $L_1 \subset L_2 \subset \ldots \subset L_{|P|}$ as subsets of $P$, with $L_{i+1} \backslash L_i$ being a single vertex for each $i$, it follows that the corresponding vectors $L_1, L_2, \ldots, L_{|P|}$ are the images of a one-to-one linear transformation of the standard basis vectors of $\R^P$ and hence linearly independent (in fact they are the extreme rays of the simplex $\omega(\mb{v}_k) \geq \omega(\mb{v}_{k-1}) \geq \ldots \geq \omega(\mb{v}_1) \geq 0$, where $\mb{v}_k = L_k \backslash L_{k-1}$). Since $\bd L_i = L_i - L_{i-1}$ as vectors, with $\bd L_1 = L_1$, and the linear transformation $(x_1, x_2, \ldots, x_n) \to (x_1, x_2 - x_1, \ldots, x_n - x_{n-1})$ is volume preserving it follows that the $\bd L_i$ are also linearly independent.
%\end{proof}

Taking the above as an algorithm for producing simplicial cones, the next step is to associate them to the maximal sets $\mc{C}(\pi)$ for $\pi \in \Pi_P$. By Theorem \ref{thm:extreme_rays} it is clear that the simplicial cone of Theorem \ref{thm:simplicial_cones} is a subset of $\mc{C}(\pi)$ iff the support of each of its extreme rays $\bd U_i$ intersects the support of $\pi$. The next result shows that there is exactly one such path $\pi$, and that we can generate it directly from the sequence of extreme rays. The basic algorithm is to start at the element of $P$ with the longest passage time and then moving backwards to the element below it with the next longest passage time, repeating until arrival at a minimal element.

\begin{theorem}\label{thm:maximal_path}
Let $U_0 \lessdot U_1 \lessdot \ldots \lessdot U_{|P|}$ be a maximal chain in $J(P)$ and $\mc{C}$ be its associated simplicial cone. For $k = 1, 2, \ldots, |P|$ let $\mb{v}_k = U_k \backslash U_{k-1}$. Define an integer-valued sequence $g_i$ by $g_1 = 1$ and
\[
g_{j+1} = \min \left \{ k : \mb{v}_k \lessdot \mb{v}_{g_j} \right \},
\]
until reaching the first integer $m$ such that $\mb{v}_{g_m}$ is a minimal element of $P$. Then the reversed subsequence $\pi_i = \mb{v}_{g_{m-i+1}}$ is a path in $\Pi_P$, and moreover $\mc{C} \subset \mc{C}(\pi)$.
\end{theorem}

\begin{proof}
That $\pi$ is a maximal chain in $\Pi_P$ is immediate from the way it is constructed as a sequence of vertices one below the next. The construction is well-defined since each vertex in $P$ appears exactly once in the sequence $\mb{v}_k$. That $\pi$ is defined as a reversed subsequence is only so that the elements along it are in increasing rather than decreasing order. Finally, to see that $\mc{C} \subset \mc{C}(\pi)$ simply observe that every extreme ray $\bd U_i$ contains (exactly) one vertex in $\pi$. This is again by construction: $\mb{v}_{g_1}$ is in $\partial U_i$ for $i < g_2$, since $U_i$ does not contain a vertex below $v_{g_1}$. For the same reason we have that $\mb{v}_{g_j} \in \bd U_i$  for $g_j \leq i < g_{j+1}$. The last element $\mb{v}_{g_m}$ is in $\bd U_i$ for $g_m \leq i \leq |P|$. This shows that every antichain $\bd U_i$ is in $\mc{C}(\pi)$, by Theorem \ref{thm:extreme_rays}, and since the maximal sets $\mc{C}(\pi)$ are clearly disjoint for different $\pi$ (up to measure zero boundaries) this shows that $\mc{C} \subset \mc{C}(\pi)$.
\end{proof}

\begin{remark}
There is also a simplicial decomposition of $\mcr{OC}(P)$, and hence of $\R_+^P$, by lower sets. This can be seen via the standard bijection between upper and lower sets, which relates the two through the bijection from upper sets to antichains and then antichains to lower sets. The decomposition works in the same way as the above, with each maximal chain in the poset of lower sets (ordered by inclusion) determining a simplex in $\mcr{OC}(P)$ which is then linearly mapped to a simplex in $\R_+^P$. In fact, the bijection between upper and lower sets also shows that there is a bijection between maximal chains of each, and it follows that a maximal chain of upper sets produces the same simplicial cone (in $\R_+^P$, not in $\mcr{OC}(P)$) as the corresponding maximal chain of lower sets. Since upper sets have a natural interpretation as passage time vectors of antichains we prefer to make the description in terms of upper sets.
\end{remark}

\begin{remark}
The maximal chains in $P$ are also in bijection with the linear extensions of $P$: the set of bijections $\sigma : P \to \{1, \ldots, |P| \}$ such that $\mb{v} \leq \mb{w} \implies \sigma(\mb{v}) \leq \sigma(\mb{w})$. The bijection is defined by letting $U_k = \{ \mb{v} \in P : \sigma(\mb{v}) \leq k \}$, with $U_0 = \emptyset$. As such the linear extensions simply correspond to the ordering of the passage time vector $(G_P(\mb{v}) : \mb{v} \in P)$, with the ordering being well-defined for Lebesgue almost all $\omega$. This implies the connection to the order growth model that we explain below.
\end{remark}

\begin{remark}
The path produced by each maximal chain of $J(P)$ is often referred to as the Sch\"{u}tzenberger or jeu-de-taquin path (see \cite{Fulton:book} for review and \cite{RomSni:jdt} for related results on infinite Young tableaux). While each maximal chain in $J(P)$ determines the maximal path for a simplicial cone of weight vectors, each given path $\pi \in \Pi_P$ is typically produced by many maximal chains of $J(P)$. Intuitively one expects that the probability that a given path is the longest one should be larger for those paths produced by more maximal chains, although this is not entirely precise because it does not take into account the probability of each simplicial cone/maximal chain under a given weight distribution. Enumerating the number of a maximal chains which produce a given path $\pi$ also appears to be difficult, even on posets of the form $[1,m] \times [1,n]$. 
\end{remark}

The proof of Theorem \ref{thm:simplicial_cones} contains the useful fact that each maximal chain in $J(P)$ induces a linear map from $\R^P$ into itself defined by $U_i \mapsto \bd U_i$ for $i=1,\ldots,|P|$. Coordinatewise the mapping works out to be of the form $\eta \mapsto \omega$ where
\[
 \omega(\mb{v}) = \eta(\mb{v}) - \max_{\mb{u} \lessdot \mb{v}} \eta(\mb{u}),
\]
if $\mb{v}$ is not a minimal element of $P$, and $\omega(\mb{v}) = \eta(\mb{v})$ if it is. It is worth recording the following important but well known observation.

\begin{lemma}\label{lem:volume_preserving}
For each maximal chain $U_0 \lessdot U_1 \lessdot \ldots \lessdot U_{|P|}$ of $J(P)$ the linear mapping of $\R^P$ to itself defined by $U_i \mapsto \bd U_i$, $i = 1, \ldots, |P|$, is volume preserving.
\end{lemma}

\begin{proof}
This follows because the matrix representing the mapping can be put into an upper triangular form with all ones on the diagonal, in the following way. Define for $k=1,\ldots,|P|$ $v_k=U_k\setminus U_{k-1}$. We take $v_1, \ldots, v_{|P|}$ as the standard basis vectors of $\R^{|P|}$. Let $A$ be the matrix defining our map in terms of our basis. Furthermore, consider the sets $U_i$ and $\partial U_i$ as vectors in $\R^{|P}$. Clearly, $Av_1=v_1$ and for $k>1$,
\begin{align*} 
Av_k & = AU_k - AU_{k-1} = v_k + \sum_{i<k} 1_{\{v_i\in \partial U_k\}} - 1_{\{v_i\in\partial U_{k-1}\}}.
\end{align*}
This shows that $A$ has 1's on the diagonal and 0's below the diagonal, proving that it is volume preserving.
\end{proof}

\subsection*{Converting the Passage Time into a Sum}

The simplicial decomposition of $\R_+^P$ by maximal chains in $J(P)$ provides a useful way of converting the passage time into a sum, in the following way. For $\omega \in \R_+^P$ let $\pi^*(\omega)$ be the (Lebesgue almost surely unique) longest path in $\Pi_P$ corresponding to $\omega$, so that
\[
 G_P = \langle \pi^*(\omega), \omega \rangle.
\]
Now for a given maximal chain $U_0 \lessdot U_1 \lessdot \ldots \lessdot U_{|P|}$ in $J(P)$ let $E = E_U$ be the $|P| \times |P|$ matrix with columns $\bd U_1, \bd U_2, \ldots, \bd U_{|P|}$, as vectors in $\R_+^P$. Then $E$ is an invertible linear map of $\R^P$ into itself (by Lemma \ref{lem:volume_preserving}), hence there exists $\Lambda_U = \Lambda \in \R^P$ such that
\begin{align}\label{eqn:lambda_to_omega}
 \omega = E_U \Lambda_U.
\end{align}
This leads to the expression
\[
 G_P = \langle \pi^*(\omega), E \Lambda \rangle = \langle E' \pi^*(\omega), \Lambda \rangle,
\]
where $E'$ denotes the transpose of $E$, so that the rows of $E'$ are antichains of $P$. Since each antichain intersects any given path at most once, this implies that $E' \pi^*(\omega)$ is a vector whose entries are either zero or one:
\[
 (E' \pi^*(\omega))_i = \1{\supp(\pi^*(\omega)) \cap \supp(\bd U_i) \neq \emptyset}.
\]
Combining these together leads to the formula
\begin{align}\label{eqn:general_sum_representation}
G_P = \sum_{i=1}^{|P|} \Lambda_i \1{\supp(\pi^*(\omega)) \cap \supp(\bd U_i) \neq \emptyset}. 
\end{align}
Note that this formula holds for \textit{any} choice of maximal chain $U$ of $J(P)$, although the value of the $\Lambda$ changes with different choices of $U$. If the $\omega$ are random but have density $f(\mb{x})$ with respect to Lebesgue measure on $\R_+^P$ then the $\Lambda$ have law $f(E \mb{x})$.

The choice of $U$ can be made depending on $\omega$. For each $\omega \in \R_+^P$ let $U(\omega)$ be the maximal chain corresponding to the passage time vector $(G_P(\mb{v}; \omega) : \mb{v} \in P)$. Then the longest path intersects each antichain $\bd U_i(\omega)$ exactly once, leading to the identity
\begin{align}\label{eqn:dynamic_lambda}
 G_P = \sum_{i=1}^{|P|} \Lambda_i.
\end{align}
In this case all $\Lambda_i$ are positive but the density is more complicated. Now it becomes the mixture
\[
 \P((\Lambda_1, \ldots, \Lambda_{|P|}) \in d \mb{x}) = \sum_{\substack{U \textrm{ maximal} \\ \textrm{chains of } J(P)}} f(E_U \mb{x}) \P(\omega \in \Span_+ \{ \bd U_1, \ldots, \bd U_{|P|} \} ).
\]

\subsection*{Corner Growth Model}

The formulas above express the last passage percolation problem in terms of the corner growth model, which is a well known equivalent description (see \cite{timo:CGM_notes, Romik:book} for reviews). In the continuous time version of corner growth the elements of $P$ are ''filled in'' at random times, subject to the constraint that an element cannot be filled in until all elements in its lower set have also been filled in. At any given time the ``corners'' are the elements of $P$ which are admissible to be filled in; this nomenclature is motivated by the process on $\Z^2$. The process starts at time zero and $\omega(\mb{v})$ is the \textrm{additional} time it takes for $\mb{v} \in P$ to be filled in \textit{after} all elements of $L(\mb{v}) \backslash \{ \mb{v} \}$ have been filled in. In the case that $L(\mb{v}) = \{ \mb{v} \}$, meaning that $\mb{v}$ is a minimal element of $P$, then there is no waiting rule and $\omega(\mb{v})$ is the time at which $\mb{v}$ is filled in. Now $G_P(\mb{v})$ is exactly the time at which element $\mb{v}$ is filled in, and if we take this as a definition of $G_P$ then it implies the recursion \eqref{eqn:recursion}. In fact this recursion shows that the corner growth description is equivalent to the last passage one. The longest path is the maximal chain $\pi$ of $P$ that takes the longest amount of time to be filled in, together with the additional requirement that for every $\mb{v} \in \pi$ this same condition holds on $L(\mb{v})$.    
 
The corner growth representation also makes clear the basic idea behind \eqref{eqn:dynamic_lambda}. The vector $G_P(\mb{v}; \omega)_{v \in P}$ is clearly in the order cone $\mcr{OC}(P)$ of $P$, due to the positivity of $\omega \in \R_+^P$. Then $G_P(\mb{v}; \omega)$ belongs to a unique (Lebesgue almost surely) simplex in $\mcr{OC}(P)$ that corresponds to a maximal chain in $J(P)$. The simplex describes the ordering according to which the elements of $P$ are filled in for this particular $\omega$, and then each $\Lambda_i$ is the time between the filling in of the $i$th and $(i-1)$st elements of $P$. More precisely, letting $U_0 \lessdot U_1 \lessdot \ldots \lessdot U_{|P|}$ be the maximal chain of $J(P)$ determined uniquely by $\omega$ (Lebesgue almost surely), we let $\mb{v}_{k} = U_k \backslash U_{k-1} \in P$ for $k = 1, 2, \ldots, |P|$ (the vertices ordered according to the time at which they appear) and then it follows that
\[
\Lambda_i = G_P(\mb{v}_i) - G_P(\mb{v}_{i-1})
\]
with $G_P(\mb{v}_0) = 0$. This clearly implies \eqref{eqn:dynamic_lambda}.

\begin{remark}
When $P$ is a Young diagram (including the poset $[1,m] \times [1,n]$) the maximal chains of $J(P)$ are in bijection with the Young tableaux for the particular diagram. The Young tableaux describes a linear map from a simplex in $\mcr{OC}(P)$ to a simplex in $\R_+^P$, with the outputted simplex being precisely the set of weights that produce that particular ordering for the passage times.
\end{remark}

\begin{comment}
 
\section{The Face Lattice \label{sec:face_lattice}}

\begin{definition}
A weight vector $\omega$ in a polyhedral cone $\mc{C}$ belongs to a $k$-dimensional face of the cone $\mc{C}$ if $\dim \mc{D}(\omega) = k$. 
\end{definition}

\begin{theorem}
Let $a_1, \ldots, a_k$ be extreme rays for a cone $\mc{C}(\pi)$ for a fixed $\pi$. Consider the associated passage time vector $H := G(a_1) + \ldots + G(a_k)$, and for $j = 1, \ldots, k$ define the antichains
\[
E_j = \partial \{ \mb{v} \in P : H(\mb{v}) \geq j \}.
\]
Form a poset $P_{a_1, \ldots, a_k}$ by connecting $E_j$ to $E_{j+1}$ through the ordering relation of the original poset $P$, for $j=1,\ldots, k-1$. Then the conical span of $a_1, \ldots, a_k$ is a $k$-dimensional face of $\mc{C}(\pi)$ iff every component connecting $E_j$ to $E_{j+1}$ is connected. 
\end{theorem}

\begin{proof}
The subposet $P_{a_1, \ldots, a_k}$ is clearly graded , with each $E_j$ acting as the $j$th level. 

To see that connectedness is necessary first consider the case $k=2$. 
\end{proof}

\begin{remark}
The passage time vector $H$ in the above theorem can be thought of as a disjoint union of the $k$ antichains $a_1, \ldots, a_k$ of $P$, but this way of thinking about it overlooks a subtle but important issue. 

This also points to another issue: that there is not a unique way of 
\end{remark}

\end{comment}

\section{Independent Exponential Weights \label{sec:exp_weights}}

The last passage model with iid exponential weights is solvable, meaning that exact formulas can be derived for various statistics such as the passage time, at least on the poset $[1,m] \times [1,n]$. One basic reason for this is that the memoryless property of the exponential distribution makes the corner growth process a continuous time Markov chain. It is a straightforward calculation to show that, when the $\omega$ are iid exponential$(1)$ random variables, at any fixed time the random amount of time until the next corner is filled has an exponential distribution with parameter equal to the number of available corners. This is made clear by the following simple fact about the exponential distribution in several variables.

\begin{lemma}\label{lem:exp_in_cone}
Let $X = (X_1, \ldots, X_n)$ with $X_i \sim \operatorname{exponential}(\lambda_i)$ independent, and let $\mb{v}_1, \ldots, \mb{v}_n \in \R_+^n$ be linearly independent. Let $\mc{C} = \Span_+ \{\mb{v}_1, \ldots, \mb{v}_n \}$. Then
\begin{align}\label{eqn:exp_in_cone_prob}
\mP(X \in \mc{C}) = |\det V| \prod_{i=1}^n \frac{\lambda_i}{\langle \mb{v}_i, \mb{\lambda} \rangle}, 
\end{align}
where $\mb{\lambda} = (\lambda_1, \ldots, \lambda_n)$ and $V$ is the $n \times n$ matrix with columns $\mb{v}_1, \ldots, \mb{v}_n$. Moreover
\[
\mc{L} \left(X \left | X \in \mc{C} \right. \right) = \mc{L} \left( \sum_{i=1}^n \Lambda_i \mb{v}_i \right) = \mc{L}(V \mb{\Lambda}),
\]   
where $\mc{L}(X | X \in \mc{C})$ denotes the conditional law of the random variable and $\Lambda_i \sim \operatorname{exponential}(\lambda_i ||\mb{v}_i||_1)$ are independent and $\mb{\Lambda} = (\Lambda_1, \ldots, \Lambda_{|P|})$.
\end{lemma}

The proof of this lemma comes by mapping $\R_+^n$ into $\mc{C}$ using the matrix $V$, but this type of mapping is one-to-one iff $\mc{C}$ is a simplex. Note that in this representation the chosen length of the extreme rays is irrelevant since it always cancels out, as in \eqref{eqn:exp_in_cone_prob}. Since we canonically take the extreme rays of the cones in Section \ref{sec:simplicial_decomp} to have entries either zero or one, it follows that the $\ell^1$ norm of any extreme ray of a maximal cone $\mc{C}(\pi)$ is equal to the number of corners available in the corner growth process. 

As in the last section the weight vector $\omega$ can always be uniquely rewritten as a linear combination of the extreme rays of the simplex that it belongs to (Lebesgue almost surely), and the simplices are in bijection with the maximal chains of $J(P)$. Lemma \ref{lem:exp_in_cone} implies that in the exponential last passage model, conditionally on the choice of cone, the coefficients in the linear combination are again independent exponentials. Alternatively, by forgetting about the conditioning one can think of the $\Lambda_i$ in \eqref{eqn:dynamic_lambda} as exponential random variables with random parameters. The parameters are independent of the exponentials and their joint law is determined by the probabilities \eqref{eqn:exp_in_cone_prob} and the $\ell_1$ lengths of the extreme rays, which as mentioned are simply counting the number of corners available at a given time. Equivalently, the joint law on parameters is determined by the directed random walk on $J(P)$ started from the minimal element $\emptyset$ and with transitions proportional the to the parameters $\lambda_i$ of the vertices available at each time. This random walk produces a random maximal chain of $J(P)$ and the vector of parameters is simply the number of corners available at each time of the walk. The difficulty in using this approach is that the sheer number of maximal chains of $J(P)$ makes it difficult to average out over the random parameters. Even on the poset $[1,m] \times [1,n]$ the distribution of the random parameters appears to be complicated.

The geometric point-of-view sheds additional some light on the unique properties of the exponential distribution. Equation \eqref{eqn:exp_in_cone_prob} can also be equated to the (suitably normalized) volume of the cone $\mc{C}$ intersected with the hyperplane $\{ \mb{x} : \langle \lambda, \mb{x} \rangle = 1 \}$, since on that part of the hyperplane the exponential density is constant. The volume description of the probability does not require that $\mc{C}$ be a simplex, and in particular implies that the probability of any given path $\pi \in \Pi_P$ being the maximal path is the (normalized) volume of the maximal cone $\mc{C}(\pi)$ intersected with the same hyperplane. This intersection is a polyhedron with codimension $1$, and while algorithms for computing its volume exist they are in general \#P hard \cite{brightwell_winkler:counting_linear_extensions}. On certain posets it may be that the structure of $\mc{C}(\pi)$ allows for more efficient computation, but in full generality it appears to be intractable. 

Since Lemma \ref{lem:exp_in_cone} only works for simplices, to compute path probabilities via this formula would require using a simplicial decomposition of $\mc{C}(\pi)$ and summing over all of the simplices. The simplicial decomposition of Section \ref{sec:simplicial_decomp} is an obvious choice but the sheer number of simplices involved makes it impractical. On the poset $[1,m] \times [1,n]$ the decomposition involves finding all Young tableaux which produce the $\pi$ as its Sch\"{u}tzenberger path. This appears to be difficult, but we expect that among all paths the extreme corner paths have the largest number of associated Young tableaux. 
\section{IID Uniform Weights \label{sec:uniform_weights}}

Independent and identically distributed uniform weights correspond to Lebesgue measure on $[0,1]^P$, and in this case the \textit{chain polytope} of Stanley \cite{stanley:two_poset_polytopes} is a useful tool in the analysis. 

\begin{definition}
The \textit{chain polytope} $\mcr{C}(P)$ of a poset $P$ is the subset of $[0,1]^P$ defined by 
\[
 \mcr{C}(P) = \{ \omega \in [0,1]^P : \mb{v}_1 < \mb{v}_2 < \ldots < \mb{v}_k \implies \omega(\mb{v}_1) + \omega(\mb{v}_2) + \ldots + \omega(\mb{v}_k) \leq 1  \}.
\]
\end{definition}

Clearly $\mcr{C}(P)$ is a bounded polytope, and by the positivity assumption on the $\omega$ it is enough to restrict the chains in the definition to just the maximal chains. Therefore $\mcr{C}(P)$ is the same as the event that $\{ G_P \leq 1 \}$ for iid uniform weights. This left-tail probability of the passage time distribution is shown in \cite[Corollary 4.2]{stanley:two_poset_polytopes} to be equal to
\begin{align}\label{eqn:vol_formula}
\P(G_P \leq 1) = \operatorname{Vol}(\mcr{C}(P)) = \frac{e(P)}{|P|!},
\end{align}
where $e(P)$ is the number of linear extensions of $P$ (equivalently the number of maximal chains in $J(P)$). For example, on the subposet $[1,m] \times [1,n]$ of $\Z^2$ this probability is
\[
 \P(G(m,n) \leq 1) = \prod_{i=1}^m \prod_{j=1}^n \frac{1}{m-i + n-j +1},
\]
thanks to the hook length formula. This particular probability is of limited use since the event $\{G_P \leq 1 \}$ is so far from the typical behavior $G(m,n) \sim c(m+n)$ as $m,n \to \infty$, but we can still use the chain polytope to give a characterization of $G_P$. In particular, we can represent the passage time $G_P$ in terms of the $\ell^1$ norm of a uniformly chosen point from a random chain polytope $\mcr{C}(P^*)$. Here $P^*$ is a random poset whose distribution is determined by $P$. It is constructed in the following way: begin with \eqref{eqn:dynamic_lambda} and let $U = U(\omega)$ be the (Lebesgue almost surely unique) maximal chain corresponding to $\omega$. Then $\omega \in \Span_+ \{ \bd U_1(\omega), \ldots, \bd U_{|P|}(\omega) \}$ implies that there are Lebesgue almost surely positive $\lambda_i(\omega)$ such that $\omega = E_U \lambda_U$, or equivalently
\[
 \omega = \sum_{i=1}^{|P|} \lambda_i(\omega) \bd U_i(\omega).
\]
But also, since $||\omega||_{\infty} \leq 1$ almost surely, this means that the $\lambda_i$ must satisfy
\[
 \left| \left| \sum_{i=1}^{|P|} \lambda_i \bd U_i \right| \right|_{\infty} \leq 1.
\]
Since the entries of each antichain $\bd U_i$ are either one or zero, this means that
\begin{align}\label{eqn:Cpstar_ineq}
 \sum_{i=1}^{|P|} \lambda_i \1{\mb{v} \in \supp(\bd U_i) } \leq 1, \quad \textrm{ for all } \mb{v} \in P.
\end{align}
Combined with the positivity condition, the set of $\lambda_i$ satisfying these inequalities is clearly the chain polytope of some poset $P^* = P^*(U)$ that is determined by the particular ordering $U$. In fact the choice of $P^*$ is not unique, but this will not concern us since we will only be concerned with the number of linear extensions of $P^*$ which turns out to be an invariant. For now let $P^*$ be any poset such that
\[
 \mcr{C}(P^*) = \left \{ \lambda_i \geq 0 : \sum_{i=1}^{|P|} \lambda_i \1{\mb{v} \in \supp(\bd U_i)} \leq 1 \textrm{ for all } \mb{v} \in P \right \}.
\]
Now $\omega$ being uniformly distributed on $[0,1]^P$ means also that its density on $\mc{C}_U \cap [0,1]^P := \Span_+ \{ \bd U_1, \ldots, \bd U_{|P|} \} \cap [0,1]^P$ is uniform. On the latter set there is a bijection between $\omega \in \mc{C}_U \cap [0,1]^P$ and $\lambda \in \mcr{C}(P^*)$ given by $\omega = E_U \lambda_U$, and since $E_U$ is volume preserving (by Lemma \ref{lem:volume_preserving}) it follows that $\lambda$ is uniformly distributed on $\mcr{C}(P^*)$, and that
\[
\Vol(\mc{C}_U \cap [0,1]^P) = \frac{e(P^*)}{|P^*|!} = \frac{e(P^*)}{|P|!}.
\]
Now via the formula
\[
 G_P = \sum_{i=1}^{|P|} \lambda_i
\]
for the passage time, this leads to a method for sampling $G_P$ when $\omega$ is iid Uniform$(0,1)$:
\begin{itemize}
 \item sample a maximal chain $U_0 \lessdot U_1 \lessdot \ldots \lessdot U_{|P|}$ according to the probabilities $e(P^*(U))/|P|!$,
 \item sample $\lambda$ as a uniform point in the associated chain polytope $\mcr{C}(P^*(U))$,
 \item return the $\ell^1$ norm $\sum_{i=1}^{|P|} \lambda_i$ as the passage time.
\end{itemize}
This characterizes $G_P$ as the $\ell^1$ norm of a point chosen uniformly from a random chain polytope, although the distribution of the polytope seems to be complicated. Evidently the final answer does not depend on the choice of $P^*$ but for the sake of concreteness we give one possible way of constructing it. By \eqref{eqn:Cpstar_ineq} the inequalities defining $\mcr{C}(P^*)$ can be written as
\[
 \sum_{i=\eta_{\mb{v}}}^{\tau_{\mb{v}}} \lambda_i \leq 1,
\]
where $\eta_{\mb{v}} = \min \{ i : \mb{v} \in \supp(\bd U_i) \}$ and $\tau_{\mb{v}} = \max \{ i : \mb{v} \in \supp(\bd U_i) \}$ are the first and last times that $\mb{v}$ is in $\supp (\bd U_i)$, respectively. Now define a graph with the times $\{1,2,\ldots, |P|\}$ as its vertices and an edge connecting $i,j$ iff there exists a $\mb{v} \in P$ such that $i,j \in [\eta_{\mb{v}}, \tau_{\mb{v}}]$. Thus an edge between two times $i$ and $j$ means that there is an element of $P$ which is in both $\bd U_i$ and $\bd U_j$, or in other words $\bd U_i \cap \bd U_j \neq \emptyset$. This graph may have multiple connected components, and each component corresponds to a component of $P^*(U)$. It can quickly be seen that a new component is born every time that there is a $\bd U_i$ with exactly one non-zero entry, so that
\[
 \textrm{number of components of } P^*(U) = 1 + | \{ i : \supp(\bd U_i) \textrm{ is a singleton} \} |.
\]
The constructed graph is the \textit{comparability graph} of the poset $P^*$. The comparability graph of any poset is defined with the elements of $P$ as its vertices and an edge connecting two vertices iff the corresponding poset elements are comparable to each other (i.e. one is in order with the other). See \cite[Chapter 5]{golumbic:book} for more. We have already encountered this object implicitly: the order graph of Definition \ref{defn:order_graph} is the comparability graph of the poset restricted to $\supp(a_1) \Delta \supp(a_2)$. Similarly, the disorder graph of Definition \ref{defn:disorder_graph} is the incomparability graph of the poset restricted to $\pi \Delta \pi'$, where the incomparability graph is defined similarly but with edges between elements that are incomparable to each other. The number of linear extensions of the poset is determined by its comparability graph, showing that the choice of $P^*$ above is irrelevant. Constructing a candidate $P^*$ is the problem of choosing a transitive orientation for the graph: an assignment of directions to the edges such that the adjacency relation of the resulting directed graph is transitive. Algorithms for finding transitive orientations are found in \cite[Chapter 5]{golumbic:book}, as are formulas for the number of transitive orientations of the graph.

\section{Open Questions \label{sec:open}}

\textbf{Most Likely Paths.} Our interest in this problem was primarily motivated by the following question: on the subposet $P = [1,n]^2$ of $\Z^2$, which up-right path has the highest probability of being the largest? Or in other words, which path $\pi \in \Pi_P$ maximizes $\P(\mc{C}(\pi))$ for a fixed probability measure $\P$ on $\R_+^P$? Even in the exactly solvable case of iid exponential weights this question appears to be difficult, for the reasons described in Section \ref{sec:exp_weights}. We are confident that the answer is the extreme corner path that goes straight up to $(1,n)$ and then straight over to $(n,n)$ (or its obvious symmetric copy across the main diagonal) but we have been unable to prove this. Certainly since the fluctuations of the path away from the main diagonal are known to be of order $n^{2/3}$ in the exactly solvable cases, and thought to be of the same order in most iid cases, the paths that are outside of the window of size $n^{1/2}$ should have more weight than the paths inside this window, which of course supports the bulk of the probability under the uniform measure on paths. 
\newline

\noindent \textbf{Negative Correlations Between Transversal Fluctuations of the Maximizer and its Length.} Again we consider the poset $[1,n]^2$ in $\Z^2$. Our belief that the outside paths have the largest individual probability of being longest is motivated by the idea that their maximal cones $\mc{C}(\pi)$ take up more of the environment space than the other maximal cones. Since the passage time $G_P$ is the inner product between the weight and the corresponding longest path, if a maximal cone $\mc{C}(\pi)$ is relatively large, then conditionally on $\omega$ being in that cone there is more room for it to point away from the path vector $\pi$. Conversely, if $\mc{C}(\pi)$ is relatively small then there is little room for an $\omega$ inside of it to point away from $\pi$. Thus we expect that there should be a \textit{negative} correlation between the length of a path and its location in the box. It would be interesting to see if this relationship effectively cancels in terms of expected values, so that each path contributes close to the same amount to the overall expected value of the passage time.
\newline

\noindent \textbf{Structure of the Path Measure.} For Lebesgue almost all $\omega \in \R_+^P$ the path lengths are all distinct, this leads to a total of $|\Pi_P|$ possible path lengths. Yet they are created from only $|P|$ random variables, leading to a strong linear dependency between the different path lengths. For example, in the case $P = [1,m] \times [1,n]$ there are a total of
\[
\dbinom{m+n}{n}
\]
different path lengths, yet it can be shown that the vector of these path lengths lives in a subspace of dimension $(m-1)(n-1) + 1$ (it is actually slightly smaller than $mn$). Consequently, every path length can be expressed as a linear combination of $(m-1)(n-1) + 1$ well chosen path lengths. We are curious to know if a similar relationship holds for the path probabilities $\P(\mc{C}(\pi))$. Does knowing $\P(\mc{C}(\pi))$ for a relatively small subset of $\pi \in \Pi_P$ determine $\P(\mc{C}(\pi))$ for all $\pi \in \Pi_P$? A result of this type would be useful as it is generally difficult to describe the measure $\mathbf{P}(\pi) := \P(\mc{C}(\pi))$ on $\Pi_P$. If one thinks of LPP as a random walk in a random environment (although really it is just a walk in a random environment), then $\mathbf{P}$ is the \textit{averaged} path measure. This $\mathbf{P}$ is not Markov, which immediately takes away one of the nicest descriptions for path measures, and is generally difficult to compute explicitly. On the other hand there should be more structure than it simply being a point in the probability simplex of dimension $|\Pi_P|$. It does not seem unreasonable to expect $\mathbf{P}$ to have some ``low-dimensional'' structure, although we are uncertain of what precisely it would be. 
\newline 

\noindent \textbf{Monotonicity of the Path Measure.} If, for square boxes $[1,n]^2$ in $\Z^2$, the outside paths have the highest probability of being the longest and the middle paths have the lowest probability, then we expect that there should some sort of monotonicity in the path probabilities as the paths move from the middle to the outside. It is not entirely clear to us what the proper ordering on paths should be, or even if a total ordering exists. We expect that there should be a natural partial ordering on the paths, based solely on their relative locations in space, such that their respective probabilities obey the partial ordering. It would be interesting to derive a rate of growth of these path probabilities along a chain in this partial ordering.
\newline

\noindent \textbf{Path Measure Proportional to the Number of Extreme Rays.} On $[1,m] \times [1,n]$ in $\Z^2$ it would be interesting to study the path measure that is proportional to the number of extreme rays per path, i.e. the measure on $\Pi_{(m,n)}$ with probabilities
\[
\mb{P}(\pi) = \frac{|\mcr{ER}(\pi)|}{\sum_{\pi} |\mcr{ER}(\pi)|}.
\]
This measure is purely combinatorial but may still exhibit many of the features of the annealed last passage measure. Given that the extreme rays of $\mcr{ER}(\pi)$ have relatively large angles between them, one might expect that the number of extreme rays of $\mcr{C}(\pi)$ is a reasonable proxy its volume, under certain measures on $\R_+^P$. If so it would give some understanding of the expected universality behavior of the paths. In particular it would be interesting to know if the transversal fluctuations of the path are superdiffusive under the annealed measure. One might also hope that the asymptotics of probabilities of certain special paths, such as the extremal ones, could be computed under this measure.
\newline

\noindent \textbf{Face Lattice of $\mc{C}(\pi)$.} The faces of a polyhedral cone are any of its intersections with half-spaces with the property that no interior point of the cone lies on the boundary of the half-space. The set of faces can be made into a lattice (in fact an Eulerian lattice), where the partial ordering is determined by set containment of faces. We have not been able to fully determine the full structure of these inclusions for $\mc{C}(\pi)$, beyond Theorem \ref{thm:edges} which explains the inclusion of the one dimensional faces (the extreme rays) in the two-dimensional faces. Being simplicial, each of the simplices $\bd U_1, \ldots, \bd U_{|P|}$ described in Section \ref{sec:simplicial_decomp} has the well known binomial poset of the appropriate size as its face lattice, but many of the faces will be interior to $\mc{C}(\pi)$. A more useful description of the face lattice would be in terms of the extreme rays of $\mc{C}(\pi)$ or the normal vectors $\pi - \pi'$ that define its facets, with the inclusions being expressed in terms of relations between these vectors. Part of our decision to explain the last passage model on general posets is motivated by a desire to explain the face lattice. Our hope is that the description of the facets can be iterated in some way to provide a description of the face lattice. Since the lower dimensional faces can be seen as ``facets of facets'', if the facets can be described as last passage percolation problems on a smaller poset, then Theorem \ref{thm:faces} can be applied again without modification. This indeed works for the facets of the form $\omega(\mb{v}) = 0$, but is more complicated for those of the form $\langle \pi - \pi', \omega \rangle = 0$. We do not know of a description of these facets as a last passage model on a smaller poset, but perhaps there is such a description in terms of a matroid or a related object.
\newline

\noindent \textbf{Number of Young Tableaux that Produce a Given Longest Path.} The simplicial decomposition of Section \ref{sec:simplicial_decomp} suffers from the curse of dimensionality, in that the number of simplices used to partition each cone $\mc{C}(\pi)$ is exponentially larger than the ambient space. The number of such simplices could perhaps be used as a rough proxy for the probability of each $\mc{C}(\pi)$. On $[1,m] \times [1,n]$ we expect that among all paths $\pi \in \Pi_P$ the extreme corner paths have the largest number of simplices in their decomposition, which would lead credence to our belief that the corner paths are the modes of the distribution. This assertion is equivalent to stating that the corner path is the Sch\"{u}tzenberger path for the largest number of Young tableaux. We would be interested in asymptotics of the number of Young tableaux that produce a given longest path. 
\newline

\noindent \textbf{Polymerization.} Closely related to the last passage percolation problem is the notion \textit{directed polymers}. Here the environment variables do not on their own determine a path, instead there is some extra randomness involved. For each inverse temperature $\beta \geq 0$ and $\omega \in \R^P$ the directed polymer measure on $\Pi_P$ is defined by the Gibbs measure
\[
\mb{P}_{\beta}^{\omega}(\pi) = e^{\beta \langle \pi, \omega \rangle}/Z_{\beta}^{\omega}
\]
where $Z_{\beta}^{\omega} = \sum_{\pi \in \Pi_P} e^{\beta \langle \pi, \omega \rangle}$ is the partition function. Note that in this case we remove the restriction that the $\omega$ have positive coordinates. With $\omega$ fixed this is often referred to as the quenched measure. As $\beta \to \infty$ it is clear that the quenched measure concentrates on the longest path (or splits uniformly amongst all paths that achieve the longest length, if there are several). In the finite temperature setting of directed polymers the analogue of the maximal cones $\mc{C}(\Pi)$ are the sets of $\omega$ which produce a given value of $\mb{P}_{\beta}^{\omega}(\pi)$, i.e. for each $\mb{Q}$ a probability measure on $\Pi_P$ one considers the set
\[
\mc{V}_{\beta}(\mb{Q}) = \left \{ \omega \in \R^P : \mb{P}_{\beta}^{\omega} = \mb{Q} \right \}.
\]
Note that for many $\mb{Q}$ this measure is likely empty, and it is an interesting problem to determine useful conditions on $\mb{Q}$ for which this is not the case. If it is not then in the variables $e^{\beta \omega(\mb{v})}$ the set is an algebraic variety, and it would be interesting if any meaningful description of it can be made. In particular, one might hope that properties of the maximal cones $\mc{C}(\pi)$ can be transferred to properties of these varieties.

\bibliographystyle{amsalpha}
\bibliography{papers}

\end{document}